\documentclass[11pt,letterpaper,reqno]{amsart}
\usepackage{amsmath}
\usepackage{amstext}
\usepackage{amssymb}
\usepackage{amsfonts}
\usepackage{enumerate}
\usepackage{amsthm}
\usepackage{mathrsfs}
\usepackage[all]{xy}
\usepackage{color}

\numberwithin{equation}{section}

\newtheorem{theorem}{Theorem}[section]
\newtheorem{lemma}[theorem]{Lemma}
\newtheorem{proposition}[theorem]{Proposition}
\newtheorem{corollary}[theorem]{Corollary}

\theoremstyle{definition}
\newtheorem{remark}[theorem]{Remark}
\newtheorem{definition}[theorem]{Definition}

\DeclareMathOperator{\Lip}{Lip}

\DeclareMathOperator{\ran}{ran}
\DeclareMathOperator{\limit}{limit}
\DeclareMathOperator{\Ann}{Ann}
\DeclareMathOperator{\Real}{Re}

\DeclareMathOperator{\interior}{int}

\newcommand{\n}[1]{ \left\|#1\right\| }
\newcommand{\bign}[1]{ \big\|#1\big\| }

\newcommand{\N}{{\mathbb{N}}}
\newcommand{\R}{{\mathbb{R}}}

\newcommand{\pair}[2]{{\langle #1, #2 \rangle}}

\newcommand{\fbs}{{\mathscr{F}}}
\newcommand{\eps}{{\varepsilon}}

\newcommand{\st}{{\; : \; }}

\renewcommand*{\L}{{\mathcal{L}}} 

\newcommand{\Lrs}[4]{{\L( #1, #2 ; #3, #4)}}
\newcommand{\Lr}[2]{{\Lrs{ #1_1}{#1_2}{#2_1}{#2_2}}}
\newcommand{\lex}{ \Lr{E}{X} }

\newcommand{\lexss}{ \Lr{E}{ X^{**} } }
\newcommand{\X}{{(X_1,X_2)}}
\newcommand{\Y}{{(Y_1,Y_2)}}
\newcommand{\J}{{(J_1,J_2)}}
\newcommand{\E}{{(E_1,E_2)}}
\newcommand{\Xss}{{(X^{**}_1,X^{**}_2)}}

\newcommand\restr[2]{{
  \left.\kern-\nulldelimiterspace 
  #1 
  \right|_{#2} 
  }}

\author[J.A.~Ch\'avez-Dom\'{\i}nguez]{Javier Alejandro Ch\'avez-Dom\'{\i}nguez}
\address{Department of Mathematics,
University of Oklahoma,
Norman OK , 73019-3103 USA}

\email{jachavezd@ou.edu}
\thanks{The author was partially supported by NSF grant DMS-1400588.}


\newcommand{\pbs}{{pair of a Banach space and a subspace}{}} 
\newcommand{\pbss}{{pairs of a Banach space and a subspace}{}}

\title{An Ando-Choi-Effros lifting theorem respecting subspaces}

\begin{document}

\maketitle

\begin{abstract}
We prove a version of the Ando-Choi-Effros lifting theorem respecting subspaces, which in turn relies on Oja's principle of local reflexivity respecting subspaces. To achieve this, we first develop a theory of pairs of $M$-ideals.
As a first consequence we get a version respecting subspaces of the Michael-Pe{\l}czy{\'n}ski extension theorem.
Other applications are related to linear and Lipschitz bounded approximation properties for a pair consisting of a Banach space and a subspace.
We show that in the separable case, the BAP for such a pair is equivalent to the simultaneous splitting of an associated pair of short exact sequences given by a construction of Lusky.
We define a Lipschitz version of the BAP for pairs, and study its relationship to the (linear) BAP for pairs.
The two properties are not equivalent in general, but they are when the pair has an additional Lipschitz-lifting property in the style of Godefroy and Kalton.
We also characterize, in the separable case, those pairs of a metric space and a subset whose corresponding pair of Lipschitz-free spaces has the BAP.
\end{abstract}

\section{Introduction}

Recall that a Banach space $X$ has the \emph{bounded approximation property (BAP)} if there exists $\lambda$ such that for each $\varepsilon>0$ and compact set $K \subset X$, there exists a finite-rank linear map $S : X \to X$ with $\n{S} \le \lambda$ and $\n{S(x)-x} \le \varepsilon$ for each $x \in K$.
Suppose that $X$ is a separable Banach space, and let $(E_n)$ be an increasing sequence of finite-dimensional subspaces of $X$ with dense union.
Define
\begin{align*}
c(E_n)&= \big\{  (x_n)  \st x_n \in E_n \text{ for each } n\in\N, \quad (x_n) \text{ converges}  \big\}\\
c_0(E_n)&= \big\{  (x_n)  \st x_n \in E_n \text{ for each } n\in\N, \quad \lim_n x_n = 0  \big\}
\end{align*}
Let us consider the short exact sequence
\begin{equation}\label{top-sequence}
\xymatrix{
0 \ar[r] & c_0(E_n)\ar[r]^j & c(E_n)\ar[r]^q & X \ar[r] &0
} 
\end{equation}
where $j$ is the natural inclusion and $q(x_n) = \lim_{n\to\infty} x_n$.
It is known that $X$ has BAP if and only if the sequence splits, that is, there is a bounded linear map $L : X \to c(E_n)$ such that $q\circ L = Id_X$ (or equivalently, a bounded linear map $R : c(E_n) \to c_0(E_n)$ such that $R \circ j = Id_{c_0(E_n)}$).
This is explicitly stated, for example, by Borel-Mathurin \cite[Lemma 2.1]{BorelMathurin}; Johnson and Oikhberg give basically the same argument in \cite[Prop. 2.4 and Remark 1]{Johnson-Oikhberg}, and the construction goes back to Lusky \cite[Sec. 3]{Lusky}. 
Analogously the Lipschitz BAP for $X$ corresponds to having a Lipschitz map $X \to c(E_n)$ splitting \eqref{top-sequence} (see \cite[Thm. 2.2]{BorelMathurin}),
and a nice way to prove both of the aforementioned characterizations is the Ando-Choi-Effros lifting theorem.
The main result of this paper is a version of said theorem respecting a subspace,
which we then use to study approximation properties for pairs via simultaneous splittings of the sequence \eqref{top-sequence} and an analogous one corresponding to a subspace. Let us recall the definition of the BAP for a pair (our definition below is not exactly the original one of Figiel, Johnson and Pe{\l}czy{\'n}ski  \cite[Defn. 1.1]{Figiel-Johnson-Pelczynski} but it is routinely checked to be equivalent;
see \cite[Thm. 4.1]{Oja-Treialt} for this and other characterizations).

\begin{definition}
If $Y$ is a subspace of a Banach space $X$,
we say that the pair $(X,Y)$ has the $\lambda$-BAP if, for each $\varepsilon>0$ and compact set $K \subset X$, there exists a finite-rank linear map $S : X \to X$ with $\n{S} \le \lambda$, $\n{S(x)-x} \le \varepsilon$ for each $x \in K$, and $S(Y) \subset Y$.
\end{definition}

Now suppose that $X$ is a separable Banach space, $Y$ is a closed subspace of $X$, $(E_n)$ is an increasing sequence of finite-dimensional subspaces of $X$ with dense union in $X$, and $(F_n)$ is an increasing sequence of finite-dimensional subspaces of $Y$ with dense union in $Y$ such that $F_n \subseteq E_n$ for each $n \in \N$.
We can extend \eqref{top-sequence} to a diagram
\begin{equation}\label{full-diagram}
	\xymatrix{
	0 \ar[r] & c_0(E_n)\ar[r] & c(E_n)\ar[r] & X \ar[r] &0\\
	0 \ar[r] & c_0(F_n)\ar[u] \ar[r] & c(F_n)\ar[u]\ar[r] & Y \ar[u]\ar[r]  &0 
	}
\end{equation}
where the rows are short exact sequences and the vertical arrows are the natural isometric embeddings.
One of our main results below (Theorem \ref{thm-BAP-equivalent-to-simultaneous-splitting}) states that the BAP for the pair $(X,Y)$ corresponds to having a  ``simultaneous lifting'' of the two short exact sequences in \eqref{full-diagram}:
that is, the existence of a map $X \to c(E_n)$ splitting the top sequence in \eqref{full-diagram} and whose restriction to $Y$ splits the bottom sequence in \eqref{full-diagram}.
This follows from Theorem \ref{thm-Ando-Choi-Effros-respecting-subspaces}, a version of the Ando-Choi-Effros Theorem where the lifting respects a subspace.
In order to even be able to state such theorem, we first develop the basics of an accompanying theory of pairs of $M$-ideals. Our main result in this regard is a characterization in terms of intersection properties of balls (Theorem \ref{thm-geo-charac}).
Another crucial ingredient is Oja's principle of local reflexivity respecting subspaces \cite{Oja-PLR-respecting-subspaces}, which we use to develop a version of Dean's identity respecting subspaces (Theorem \ref{thm-PLRRS-a-la-Dean}).

Once we have the Ando-Choi-Effros theorem respecting subspaces, 
as a first consequence we get a version respecting subspaces of the Michael-Pe{\l}czy{\'n}ski extension theorem.
Another one will be the aforementioned characterization of BAP for pairs in terms of ``simultaneous liftings'' for \eqref{full-diagram}.

Afterwards we study a Lipschitz version of the BAP for pairs.
We show that in full generality the version for pairs of the Godefroy-Kalton theorem (that is, the equivalence of BAP and Lipschitz BAP) does not hold.
Nevertheless the equivalence does hold if we introduce an additional hypothesis, a Lipschitz-lifting property respecting subspaces that is always satisfied by Lipschitz-free spaces.
In the last section, we study pairs of a metric space space and a subset whose corresponding pair of Lipschitz-free spaces has the BAP.
In the case of compact metric spaces, we use ideas of Godefroy \cite{Godefroy-Extensions-and-BAP} together with our  Ando-Choi-Effros Theorem respecting subspaces to get a characterization in terms of near-extension operators.
For the more general case of separable metric spaces, we adapt the strategy of Godefroy and Ozawa  \cite{Godefroy-Ozawa} to prove a characterization in terms of a Lipschitz version of the local reflexivity principle respecting subspaces.

\section{Notation and preliminaries}

We say that $\X$ is a \emph{\pbs} if $X_1$ is a Banach space and $X_2 \subseteq X_1$ is a closed subspace of $X_1$.
For two \pbss  {} $\J$ and $\X$, we write $\J \subseteq \X$ when $J_1 \subseteq X_1$ and  $J_2 \subseteq X_2$.
In such a situation, the annihilators $J_1^\perp, X_2^\perp$ (resp. $J_1^{\perp\perp}, X_2^{\perp\perp}$ ) are always taken in $X_1^*$ (resp. $X_1^{**}$) whereas $J_2^\perp$ (resp. $J_2^{\perp\perp}$) is always taken in $X_2^*$ (resp. $X_2^{**}$).
In the rare occasions where we need to specify a different annihilator, we use the notation
$$\Ann(J,X^*) = \{ x^* \in X^* \st x^*(x) = 0 \text{ for all } x\in J \}.$$
If $A_2 \subset A_1$ and $B_2 \subset B_1$, we write $f : (A_1,A_2) \to (B_1,B_2)$ to mean that $f : A_1 \to B_1$ is a function such that $f(A_2) \subset B_2$.
If $\X$ and $\Y$ are \pbss{}, and $T : \X \to \Y$ is a linear map, 
we say that $T$ is a \emph{simultaneous (linear) isomorphism} if $T$ is a (bounded linear) isomorphism and $T(X_2) = Y_2$.
If instead we assume that $T$ is a Lipschitz bijection with Lipschitz inverse, that will be called a \emph{simultaneous Lipschitz isomorphism}.
By a \emph{paving} of a separable Banach space $X$ we mean an increasing sequence $(E_n)$ of finite-dimensional subspaces of $X$ whose union is dense in $X$.
A paving of a \pbs{} $\X$ is a sequence of \pbss{} $(E_n,F_n)$ such that $(E_n)$ is a paving for $X_1$, $(F_n)$ is a paving for $X_2$ and $F_n \subseteq E_n$ for all $n\in\N$.

For a point $x$ in a Banach space $X$ and $r \ge 0$, we denote by $B(x,r)$ the closed ball
$\{ y \in X \st \n{x-y} \le r \}$.
If the space $X$ needs to be emphasized, we write $B_X(x,r)$. The closed unit ball $B_X(0,1)$ is denoted simply by $B_X$.

For a closed subset $D$ of a compact Hausdorff space $K$, we always denote
$$
J_D = \{ f \in C(K) \st f(x) = 0 \text{ for all } x \in D\}.
$$

All maps between Banach spaces are assumed to be bounded linear maps, unless explicitly stated otherwise.
Let $X$ and $Y$ be Banach spaces.
The space of all bounded linear maps from $X$ to $Y$ is denoted by $\L(X;Y)$.
We say that a bounded linear map $T : X \to Y$ is a \emph{metric surjection} if the image of the open unit ball of $X$ is the open unit ball of $Y$.
We write $X \equiv Y$ to indicate that $X$ and $Y$ are isometrically isomorphic.

A \emph{short exact sequence} is $\xymatrix{0 \ar[r] & Z \ar[r]^A & Y\ar[r]^B &X \ar[r] &0}$ where the image of each arrow coincides with the kernel of the following one. We say that a short exact sequence is \emph{isometric} if $A$ is an isometric embedding and $B$ is a metric surjection. 

Let $\X$ and $\Y$ be \pbss{}, and $T : \X \to \Y$ a bounded linear map. 
A linear map $R : \Y \to \X$ is called a \emph{simultaneous retraction for $(T,\restr{T}{X_2})$} if $R \circ T = Id_{X_1}$
(note that in this case, $\restr{R}{Y_2} \circ \restr{T}{X_2} = Id_{X_2}$).
A linear map  $L : \Y \to \X$ is called a \emph{simultaneous lifting for $(T,\restr{T}{X_2})$} if $T \circ L = Id_{Y_1}$
(note that in this case, $\restr{T}{X_2} \circ \restr{L}{Y_2}  = Id_{Y_2}$).
If the map $L$ is only assumed to be Lipschitz, it is called a \emph{simultaneous Lipschitz lifting}.
If the maps $(T,\restr{T}{X_2})$ are clear from the context, we will not mention them explicitly.

\begin{definition}
By a \emph{pair of short exact sequences} we mean a commutative diagram of the form
\begin{equation}\label{eqn-pair-exact-sequences}
\xymatrix{
	0 \ar[r] & Z_1 \ar[r]^{A_1} & Y_1 \ar[r]^{B_1} &X_1 \ar[r]  & 0\\
	0 \ar[r] & Z_2 \ar[r]^{A_2}\ar[u]^R & Y_2 \ar[r]^{B_2}\ar[u]^{S} &X_2 \ar[r] \ar[u]^T  & 0\\
	}
\end{equation}
where every row is exact.

If $R$ and $S$ are isometric embeddings, we say that \eqref{eqn-pair-exact-sequences} is \emph{left-isometric}.
In this case, we say that \eqref{eqn-pair-exact-sequences} admits a simultaneous retraction if there is a simultaneous retraction for $(A_1,A_2)$.
If $S$ and $T$ are isometric embeddings, we say that \eqref{eqn-pair-exact-sequences} is \emph{right-isometric}.
In this case, we say that \eqref{eqn-pair-exact-sequences} admits a simultaneous (Lipschitz) lifting, if there is a simultaneous (Lipschitz) lifting for $(B_1,B_2)$.
If \eqref{eqn-pair-exact-sequences} is both left- and right-isometric, we will simply say it is isometric.
\end{definition}

\begin{remark}\label{sim-lifting-iff-retraction}
The usual arguments show that an isometric pair of short exact sequences admits a simultaneous retraction if and only if it admits a simultaneous lifting; we will call either of these a \emph{simultaneous splitting}.
\end{remark}

Let $X$ be a Banach space. A linear projection $P : X \to X$ is called an \emph{$M$-projection} (resp. \emph{$L$-projection}) if for all $x \in X$ we have $\n{x} = \max\{ \n{Px}, \n{x-Px} \} $ (resp.$ \n{x} = \n{Px} + \n{x-Px}$).
A closed subspace $J \subset X$ is called an \emph{$M$-summand} (resp. \emph{$L$-summand}) if it is the range of an $M$-projection (resp. $L$-projection),
and it is called an \emph{$M$-ideal} if $J^\perp$ is an $L$-summand in $X^*$.
For the general theory of $M$-ideals in Banach spaces, we refer the reader to \cite{Harmand-Werner-Werner}.

We use the convention of having \emph{pointed} metric spaces, i.e. with a designated special point always denoted by $0$.
For a metric space $M$ and a Banach space $X$,
$\Lip_0(M;X)$ denotes the Banach space of Lipschitz functions $T : M \to X$ such that $T(0)=0$, with addition defined pointwise and the Lipschitz constant $\Lip(T)$ as the norm of $T$.
When $X = \R$, we simply write $\Lip_0(M)$ or $M^\#$.
The \emph{Lipschitz-free space} of a metric space $M$, denoted $\fbs(M)$, is the canonical predual of $\Lip_0(M)$, that is, the closed linear subspace of $\Lip_0(X)$ spanned by the evaluation functionals $\delta(x) : f \mapsto f(x)$ for $f \in \Lip_0(M)$ and $x \in M$.
The map $\delta : x \mapsto \delta(x)$ is an isometric embedding of $M$ into $\fbs(M)$. Moreover,
for any Banach space $X$ and any Lipschitz map $T : M \to X$ with $T(0)=0$ there is a unique linear map $\overline{T} : \fbs(M) \to X$ such that $\overline{T} \circ \delta = T$. Furthermore, $\bign{\overline{T}} = \Lip(T)$.
It is because of this universal property that the space $\fbs(X)$ is called the Lipschitz-free space of $M$, or simply the free space of $M$.
This concept goes back to \cite{Arens-Eells-56}, see \cite{Weaver} for a thorough study.
Lipschitz-free spaces have been recently used as tools in nonlinear Banach space theory, see \cite{Godefroy-Kalton-03,Kalton-04} and the survey \cite{Godefroy-Lancien-Zizler}.
In this context, the non-linear map $\delta$ has a linear left inverse \cite[Lemma 2.4]{Godefroy-Kalton-03}:
if $\mu$ is a measure with finite support on the Banach space $X$, we can define its barycenter as $\beta(\mu) = \int x d\mu(x)$; since such measures can be identified with a dense subset of $\fbs(X)$, $\beta$ extends to a norm-one linear operator $\beta_X : \fbs(X) \to X$ that we call the \emph{barycentric map}.

\section{Simultaneous $M$-ideals}

In order to state our version of the Ando-Choi-Effros theorem respecting subspaces, we need to develop a corresponding theory of $M$-ideals. In this case the subspaces will not be respected by the $M$-ideals, but rather by the associated $L$- and $M$-projections. In order to avoid awkward terminology, we have then chosen to use the word ``simultaneous'' when talking about the ideals.

\begin{definition}
Let $\X$ be a \pbs. We say that a linear projection $P$ on $X_1$ is a \emph{simultaneous $M$- (resp. $L$-) projection for $\X$} if
$P:X_1 \to X_1$ is an $M$- (resp. $L$-) projection and $P(X_2) \subseteq X_2$.
\end{definition}

Our choice of terminology is justified by the fact that in this case  $\restr{P}{X_2} : X_2 \to X_2$ is clearly also an $M$- (resp. $L$-) projection.

\begin{definition}
Let $\J \subseteq \X$ be \pbss.
\begin{enumerate}[(a)]
\item
We say that $\J$ is a \emph{simultaneous $M$- (resp. $L$-)  summand in $\X$} if there is a simultaneous $M$- (resp. $L$-) projection $P: X_1 \to X_1$ for $\X$ with $P(X_1) = J_1$ and $P(X_2) = J_2$. 
\item
We say that $\J$ is a \emph{simultaneous $M$-ideal for $\X$} if $J_1$ is an $M$-ideal in $X_1$ and $J_2$ is an $M$-ideal in $X_2$, with associated $L$-projections $Q_1 : X_1^* \to J_1^\perp$ and $Q_2 : X_2^* \to J_2^\perp$ such that the diagram
\begin{equation}\label{diagram-simultaneous-M-ideals-L-projections}
	\xymatrix{
	X_1^* \ar[r]^{Q_1} \ar[d]_{r} & J_1^\perp \ar[d]^{\restr{r}{J_1^\perp}} \\
	X_2^* \ar[r]^{Q_2} & J_2^\perp 
	}
\end{equation}
commutes, where $r : X_1^* \to X_2^*$ is the restriction map.
In this case, we say that the projections $Q_1$ and $Q_2$ are \emph{compatible}.
\end{enumerate}
\end{definition}

\begin{remark}\label{remark-simultaneous-Mideal}
We note two important consequences that will be repeatedly used in the sequel:
\begin{enumerate}[(a)]
\item
For $k=1,2$, by taking $P_k = (Id_{X_k^*} - Q_k)^*$, we get an $M$-projection with range $J_k^{\perp\perp}$.
The diagram \eqref{diagram-simultaneous-M-ideals-L-projections} gives rise to another commutative diagram
\begin{equation}\label{diagram-simultaneous-M-ideals-M-projections}
	\xymatrix{
	X_1^{**}\ar[r]^{P_1}  & J_1^{\perp\perp}  \\
	X_2^{**} \ar[r]^{P_2}\ar[u]^{\iota} & J_2^{\perp\perp} \ar[u]_{\restr{\iota}{J_2^{\perp\perp}}}
	}
\end{equation}
where $\iota : X_2^{**} \to X_1^{**}$ is the canonical inclusion.
Thus there is always a simultaneous $M$-projection associated to a simultaneous $M$-ideal, and in fact $(J_1^{\perp\perp},J_2^{\perp\perp})$ is a simultaneous $M$-summand in $(X_1^{**},X_2^{**})$.

\item
When $\J$ is a simultaneous $M$-ideal for $\X$, there is a canonically associated pair of short exact sequences:
\begin{equation}\label{diagram-simultaneous-M-ideals-isometric-embeddings}
	\xymatrix{
	0 \ar[r] & J_1 \ar[r]^{j_1} & X_1 \ar[r]^{q_1} & X_1/J_1 \ar[r] &0\\
	0 \ar[r] & J_2 \ar[u]^{\iota} \ar[r]^{j_2} & X_2 \ar[u]^{\iota'}\ar[r]^{q_2} & X_2/J_2 \ar[u]^T\ar[r]  &0 
	}
\end{equation}
where $\iota$, $\iota'$, $j_1$ and $j_2$ are the natural inclusion maps (in particular, they are isometric embeddings) and $q_1$, $q_2$ are the canonical quotient maps. $T$ is simply the linear map defined on $X/2/J_2$ which is induced by $q_1 \circ \iota'$, this is well-defined since $J_2 \subset \ker(q_1 \circ \iota')$.
Let us now observe that $T$ is in fact an isometric embedding. This will be the case if and only if its adjoint $(X_1/J_1)^* \to (X_2/J_2)^*$ is a metric surjection, but this map can be canonically identified with the restriction map $\restr{r}{J_1^\perp} : J_1^\perp \to J_2^\perp$. The fact that the latter is a metric surjection follows immediately from the diagram \eqref{diagram-simultaneous-M-ideals-L-projections}, together with the fact that $r : X_1^* \to X_2^*$ is a metric surjection and so is any $L$-projection.
We will call \eqref{diagram-simultaneous-M-ideals-isometric-embeddings} the pair of short exact sequences associated to the simultaneous $M$-ideal.
\end{enumerate}
\end{remark}

The main result of this section is a characterization of simultaneous $M$-ideals in terms of intersection properties, which has the usual advantage of dispensing with the associated $L$-projections. We start by recording a Lemma, a slight refinement of \cite[Lemma I.2.1]{Harmand-Werner-Werner}. The proof is exactly the same, so we omit it.

\begin{lemma}\label{lemma-intersection-M-summands}
Suppose $\J$ is a simultaneous $M$-summand in $\X$, and suppose that $x_1, \dotsc, x_n$ in $X_1$ and positive numbers $r_1, \dotsc, r_n$ satisfy
$$
B(x_i, r_i) \cap J_1 \not= \emptyset \text{ for each }i=1,\dotsc,n \qquad \text{and} \qquad  \bigcap_{i=1}^n B(x_i, r_i) \cap X_2 \not= \emptyset.
$$
Then $\bigcap_{i=1}^n B(x_i, r_i) \cap J_2 \not= \emptyset$.
\end{lemma}

We are now ready for the promised characterization, modeled after \cite[Thm. I.2.2]{Harmand-Werner-Werner}.

\begin{theorem}\label{thm-geo-charac}
Let $\J \subseteq \X$ be \pbss, and assume that $J_1$ is an $M$-ideal in $X_1$.
The following are equivalent:
\begin{enumerate}[(i)]
\item \label{thm-geo-Mideal} 
$\J$ is a simultaneous $M$-ideal in $\X$.
\item \label{thm-geo-nball}
For all $n\in\N$, whenever  $x_1, \dotsc, x_n$ in $X_1$ and positive numbers $r_1, \dotsc, r_n$ satisfy
$$
B(x_i, r_i) \cap J_1 \not= \emptyset \text{ for each }i=1,\dotsc,n \qquad \text{and} \qquad  \bigcap_{i=1}^n B(x_i, r_i) \cap X_2 \not= \emptyset,
$$
it follows that for every $\eps>0$
$$
\bigcap_{i=1}^n B(x_i, r_i+\eps) \cap J_2 \not= \emptyset.
$$
\item \label{thm-geo-3ball}
Same as \eqref{thm-geo-nball} with $n=3$.
\item \label{thm-geo-restricted-3ball}
For all $y_1,y_2,y_3 \in B_{J_1}$, all $x \in B_{X_2}$ and $\eps>0$ there is $y \in J_2$ satisfying
$$
\n{x + y_i - y} \le 1+\eps\qquad \text{ for }i=1,2,3.
$$
\item \label{thm-geo-strict-nball}
For all $n\in\N$, whenever  $x_1, \dotsc, x_n$ in $X_1$ and positive numbers $r_1, \dotsc, r_n$ satisfy
$$
B(x_i, r_i) \cap J_1 \not= \emptyset \text{ for each }i=1,\dotsc,n \qquad \text{and} \qquad \interior \Big( \bigcap_{i=1}^n B(x_i, r_i) \Big) \cap X_2 \not= \emptyset,
$$
it follows that for every $\eps>0$
$$
\bigcap_{i=1}^n B(x_i, r_i+\eps) \cap J_2 \not= \emptyset.
$$
\end{enumerate}

\end{theorem}

\begin{proof}
\eqref{thm-geo-Mideal} $\Rightarrow$ \eqref{thm-geo-nball}:
By Remark \ref{remark-simultaneous-Mideal}, $(J_1^{\perp\perp}, J_2^{\perp\perp})$ is a simultaneous $M$-summand in $(X_1^{**},X_2^{**})$.
Consider the corresponding balls $B_{X_1^{**}}(x_i,r_i)$ in $X_1^{**}$; Lemma \ref{lemma-intersection-M-summands} allows us to find
$$
x_0^{**} \in \bigcap_{i=1}^n B_{X_1^{**}}(x_i, r_i) \cap J_2^{\perp\perp}.
$$
The rest of the proof continues as in that of \cite[Thm. I.2.2]{Harmand-Werner-Werner}.

\eqref{thm-geo-nball}  $\Rightarrow$ \eqref{thm-geo-3ball}:
This is just the specialization to the case $n=3$.

\eqref{thm-geo-3ball} $\Rightarrow$ \eqref{thm-geo-restricted-3ball}:
This is just the special case $x_i = x + y_i$, $r_i=1$.
 
\eqref{thm-geo-restricted-3ball} $\Rightarrow$\eqref{thm-geo-Mideal}:
For $k = 1,2$ define
$$
J_k^\# = \{ x^* \in X_k^* \st \n{x^*} = \bign{ \restr{x^*}{J_k} } \}.
$$
From the proof of \cite[Thm. I.2.2]{Harmand-Werner-Werner} it follows that each $x^* \in X_k^*$ can be written in a unique way as $x^* = u^*_{k} + v^*_{k}$ with $u^*_{k} \in J_k^\perp$ and $v^*_{k} \in J_k^\#$, and moreover the map $P_k : x^* \mapsto u^*_{k}$ is an $L$-projection from $X_k^*$ onto $J_k^\perp$.
All that is left to check is that $P_1$ and $P_2$ are compatible.

Define also
$$
J^\# =  \{ x^* \in X_1^* \st \n{x^*} = \bign{ \restr{x^*}{J_2} } \} \subseteq J_1^\#,
$$
and note that any $x^* \in X_1^*$ can be written as $x^* = a^* + b^*$ with $a^* \in \Ann(J_2,X_1^*)$ and $b^* \in J^\#$ (simply take $b^*$ to be a Hahn-Banach extension of $\restr{x^*}{J_2}$, and set $a^* = x^*-b^*$).

Fix $x^* \in X_1^*$, and write it as $x^* = u^* + v^*$ with $u^* \in J_1^\perp$ and $v^* \in J_1^\#$ (that is, $u^* = P_1x^*$).
Now write $v^* = a^* + b^*$ with $a^* \in \Ann(J_2,X_1^*)$ and $b^* \in J^\#$.
Let $x \in B_{X_2}$. Given $\eps>0$, choose $y_1,y_2 \in B_{J_1}$ satisfying
\begin{align*}
\R \ni \pair{v^*}{y_1} &\ge \n{\restr{v^*}{J_1}} - \eps = \n{v^*} - \eps\\
\R \ni -\pair{b^*}{y_2} &\ge \n{\restr{b^*}{J_1}} - \eps = \n{b^*} - \eps\\
\end{align*}
(note we have used that $b^*,v^* \in J_1^\#$).
Using \eqref{thm-geo-restricted-3ball} we can get $y \in J_2$ such that
$$
\n{x + y_i - y} \le 1 +\eps, \quad i=1,2.
$$
Now,
\begin{align*}
(1+\eps)\big( \n{v^*} + \n{b^*} \big) &\ge \big| \pair{v^*}{x+y_1-y} - \pair{b^*}{x+y_2-y} \big| \\
&= \big| \pair{v^*-b^*}{x} + \pair{v^*}{y_1} - \pair{b^*}{y_2} - \pair{v^*- b^*}{y} \big| \\
&\ge \Real \pair{a^*}{x} + \n{v^*} + \n{b^*} - 2\eps
\end{align*}
Letting $\eps$ go to zero we conclude $\pair{a^*}{x} = 0$ for all $x \in X_2$.

Since $x^* = u^* + a^* + b^*$, it follows that
$$
\restr{x^*}{X_2} = \restr{u^*}{X_2} + \restr{a^*}{X_2} + \restr{b^*}{X_2} = \restr{u^*}{X_2} + \restr{b^*}{X_2}.
$$
From $u^* \in J_1^\perp$ it follows that $ \restr{u^*}{X_2} \in J_2^\perp$, and from $b^* \in J^\#$ it follows that $\restr{b^*}{X_2} \in J_2^\#$.
Therefore, $P_2\big( \restr{x^*}{X_2}  \big) = \restr{u^*}{X_2} = \restr{(P_1x^*)}{X_2}$, so the projections are compatible.

\eqref{thm-geo-nball}  $\Leftrightarrow$ \eqref{thm-geo-strict-nball}:
The proof is a straightforward adaptation of the argument in \cite[Thm. I.2.2]{Harmand-Werner-Werner}; we leave out the details since we will not be making use of characterization \eqref{thm-geo-strict-nball} in the rest of this paper.
\end{proof}

\begin{remark}
Though the proof above for Theorem \ref{thm-geo-charac} might give the impression of only making use of \eqref{thm-geo-nball} in the case $n=2$, the case $n=3$ was implicitly used for the implication \eqref{thm-geo-restricted-3ball} $\Rightarrow$\eqref{thm-geo-Mideal} when citing the proof of \cite[Thm. I.2.2]{Harmand-Werner-Werner}.
As pointed out in \cite[Remarks I.2.3]{Harmand-Werner-Werner}, the case $n=2$ is enough to get a nonlinear ``$L$-projection'', but $n=3$ is required in order to get the linearity.
\end{remark}

As a consequence of Theorem \ref{thm-geo-charac}, we get our first (and very useful) example of a simultaneous $M$-ideal.

\begin{corollary}\label{cor-example-simultaneous-c0}
Suppose that $(X,Y)$ is a \pbs{}, with $X$ separable, and let $(E_n,F_n)$ be a paving of $(X,Y)$.
Then $\big( c_0(E_n), c_0(F_n)\big)$ is a simultaneous $M$-ideal in $\big( c(E_n), c(F_n)\big)$.
\end{corollary}

\begin{proof}
It is well-known that $c_0(E_n)$ is an $M$-ideal in $c(E_n)$, see for example the proof of  \cite[Prop. II.2.3]{Harmand-Werner-Werner}.
An easy adaptation of that argument will give the rest:
given sequences of norm at most one $x = (x_n) \in c(F_n)$ and $y_i = (y_n^{i}) \in c_0(E_n)$ ($i=1,2,3$), and $\eps>0$, choose $N\in\N$ such that $\n{y_n^i} \le \eps$ for $n \ge N$ and $i=1,2,3$.
If we define $y = (y_n)$ by $y_n = x_n$ for $n \le N$ and $y_n = 0$ for $n>N$, it follows that $y \in c_0(F_n)$ and
$\n{x + y_i - y} \le 1+\eps$.
Theorem \ref{thm-geo-charac} now gives the desired conclusion.
\end{proof}

\begin{remark}
The example in Corollary \ref{cor-example-simultaneous-c0} shows that in order for $\J$ to be a simultaneous $M$-ideal in $\X$, it is not necessary to have $J_2 = J_1 \cap X_2$.
\end{remark}

 Before we present our next example of a simultaneous $M$-ideal, we recall a definition.
 
\begin{definition}
Let $K$ be a compact Hausdorff space.
Suppose $X$ is a closed subspace of $C(K)$, and $D \subseteq K$ is closed.
We define $\restr{X}{D}$ to be the space of all restrictions $\{ \restr{f}{D} \st f \in X\}$.
We say that $(\restr{X}{D}, X)$ has the \emph{bounded extension property} \cite{Michael-Pelczynski} if there exists a constant $C$ such that, given $f \in \restr{X}{D}$, $\eps>0$ and an open set $U \supset D$, there is some $F \in X$ such that
$$
\restr{F}{D} = f, \qquad \n{F} \le C \n{f}, \qquad | f(x) | \le \eps \text{ for } x \not\in U.
$$
 \end{definition}

\begin{corollary}\label{cor-joint-bounded-extension-property}
Let $K$ be a compact Hausdorff space.
Suppose $X_2 \subseteq X_1$ are closed subspaces of $C(K)$, and $D \subseteq K$ is closed.
If $(\restr{X_1}{D}, X_1)$ and $(\restr{X_2}{D}, X_2)$ both have the bounded extension property,
then $(J_D\cap X_1, J_D \cap X_2)$ is an $M$-ideal in $(X_1,X_2)$.
\end{corollary}

\begin{proof}
It follows from \cite[Prop. I.1.20]{Harmand-Werner-Werner} that $J_D \cap X_1$ is an $M$-ideal in $X_1$.
A straightforward adaptation of the argument in \cite[Proof of Prop. I.1.20, p. 24]{Harmand-Werner-Werner} shows that condition \eqref{thm-geo-restricted-3ball} in Theorem \ref{thm-geo-charac} is satisfied, and the conclusion follows.
\end{proof}

Our next example generalizes \cite[Ex. VI.4.1]{Harmand-Werner-Werner}

\begin{corollary}
\begin{enumerate}[(a)]
\item
Let $1 < p \le q < \infty$, and let $X \subset \ell_q$ be a subspace isometric to $\ell_q$.
Then $\big( \mathcal{K}(\ell_p, \ell_q), \mathcal{K}(\ell_p, X) \big)$ is a simultaneous $M$-ideal in $\big( \mathcal{L}(\ell_p, \ell_q), \mathcal{L}(\ell_p, X) \big)$.
\item Let $Y$ be any Banach space, and $X \subset c_0$ a subspace isometric to $c_0$.
Then $\big( \mathcal{K}(Y, c_0), \mathcal{K}(Y, X) \big)$ is a simultaneous $M$-ideal in $\big( \mathcal{L}(Y,c_0), \mathcal{L}(Y,X) \big)$.
\end{enumerate}
\end{corollary}

\begin{proof}
(a) From \cite[Ex. VI.4.1]{Harmand-Werner-Werner}, we have that $\mathcal{K}(\ell_p, \ell_q)$ is an $M$-ideal in $\mathcal{L}(\ell_p, \ell_q)$, so we can apply Theorem \ref{thm-geo-charac}.
Let $S_1, S_2, S_3 \in \mathcal{K}(\ell_p, \ell_q)$ and $T \in \mathcal{L}(\ell_p, X)$ be contractions.
Let $(P_n)$ (resp. $(Q_n)$) be the sequence of projections associated to the unit vector basis of $\ell_p$ (resp. $X$, thought of as $\ell_q$).
Let $\varepsilon>0$ be given.
We will show that for $n,m$ large enough,
$$
\n{ T+ S_i - ( Q_nT - TP_m + Q_n T P_m ) } \le 1 + \varepsilon \qquad \text{ for } i=1,2,3,
$$
which will prove that condition \eqref{thm-geo-restricted-3ball} in Theorem \ref{thm-geo-charac} is satisfied since $Q_nT - TP_m + Q_n T P_m$ is a compact operator on $\ell_p$ with values in $X$.
The rest of the proof goes exactly as in \cite[Ex. VI.4.1]{Harmand-Werner-Werner}.

The proof for (b) is similar but easier: in this case one shows $\n{T+S_i-Q_nT} \le 1 + \varepsilon$ where $(Q_n)$ are the projections associated to the unit vector basis of $X$ (after we identify $X$ with $c_0$).
\end{proof}

\section{Oja's Principle of Local Reflexivity respecting subspaces, \`a la Dean}

Dean's version of the Principle of Local Reflexivity \cite{Dean73} asserts that when $E$ and $X$ are Banach spaces with $E$ finite-dimensional, then $\L(E;X)^{**} \equiv \L(E;X^{**})$ with the identification given by
\begin{equation}\label{eqn-dean-identification}
\varphi \mapsto \big[ \tilde{\varphi} :   e \mapsto \big[ x^* \mapsto   \varphi( e \otimes x^* ) \big] \big].
\end{equation}
Before proving a version respecting subspaces, we need to define the appropriate space of operators.

\begin{definition}
Let $\E$ and $\X$ be \pbss. We define
$$
\lex = \big\{ S \in \L(E_1,X_1) \mid S(E_2) \subseteq X_2 \big\}
$$
Note that this is a closed subspace of $\L(E_1;X_1)$.
\end{definition}

Now we proceed to the main result of this section, a version of Dean's identity respecting subspaces based on Oja's Principle of Local Reflexivity respecting subspaces.

\begin{theorem}\label{thm-PLRRS-a-la-Dean}
Let $\E$ and $\X$ be \pbss, with $E_1$ finite-dimensional.
Then $\lex^{**} \equiv \lexss$, with the identification given by \eqref{eqn-dean-identification}.
\end{theorem}

\begin{proof}
By Dean's result, $\L(E_1;X_1)^{**} \equiv \L(E_1;X_1^{**})$.
Since $\lex$ is a subspace of $\L(E_1;X_1)$,
$$
\lex^{**} \equiv \lex^{\perp\perp} \subseteq\L(E_1;X_1)^{**}  \equiv  \L(E_1; X_1^{**}).
 $$
Let $\varphi \in \lex^{**}$. We can consider it as a map $\tilde{\varphi} \in \L(E_1; X_1^{**})$, and
moreover $\pair{\varphi}{R^*} = 0$ for any $R^* \in  \lex^{\perp}$.
Let $e \in E_2$ and $x^* \in X_2^\perp \subseteq X_1^*$.
Note that for any $S \in \lex$, since $Se \in X_2$,
$$
\pair{e \otimes x^*}{S} = \pair{Se}{x^*} = 0.
$$
Therefore $e \otimes x^* \in \lex^{\perp}$, and hence
$$
0 = \pair{\varphi}{e \otimes x^*} = \pair{\tilde{\varphi}e}{x^*}.
$$
This shows that for any $e \in E_2$, $\tilde{\varphi}e \in X_2^{\perp\perp} \equiv X_2^{**}$; that is, $\tilde{\varphi} \in \lexss$.
\\
\\
Conversely, assume that we have $\tilde{\varphi} \in \lexss$.
By the Principle of Local Reflexivity respecting subspaces \cite[Thm. 1.2]{Oja-PLR-respecting-subspaces}, for every $\alpha=(\eps,F)$ where $\eps>0$ and $F$ is a finite-dimensional subspace of $X_1^*$,
there exists $S_\alpha \in \lex$ such that
\begin{enumerate}[(i)]
\item $ \n{S_\alpha} \le  (1+\eps) \n{\tilde{\varphi}}$.
\item $\pair{ \tilde{\varphi}e  }{y^*}  = \pair{S_\alpha e}{y^*}$ for every $e \in E$ and $y^* \in F$.
\item$\tilde{\varphi}e = S_\alpha e$ for those $e \in E$ for which $\tilde{\varphi}e \in X_1$.
\end{enumerate}

Now fix $R^* \in  \lex^{\perp}$.
Note that $R^*$, as an element of $\L(E_1;X_1)^*$, defines a map in $\L(E_1^*,X_1^*)$ given by $e^* \mapsto \big[ x \mapsto \pair{R^*}{ e^* \otimes x } \big]$.
Let $F \subseteq X^*$ be the range of $R^*$ considered as the map above.

Note that when $\alpha$ is large enough we have that $\pair{S_\alpha}{R^*} = \pair{ \tilde{\varphi} }{R^*}$, because if we write $R^* = \sum_{j=1}^n e_j \otimes x_j^*$ with $e_j \in E$ and $x_j^* \in X_1^*$, then
$$
\pair{ \tilde{\varphi} }{R^*} = \pair{ \tilde{\varphi} }{\sum_{j=1}^n e_j \otimes x_j^*} = \sum_{j=1}^n \pair{ \tilde{\varphi}e_j }{ x_j^*}
$$
and the latter is equal to $\sum_{j=1}^n \pair{ S_\alpha e_j }{ x_j^*} = \pair{S_\alpha}{R^*}$ for $\alpha$ large enough.
But this is equal to zero because $S_\alpha \in \lex$ and $R^* \in \lex^\perp$,
so we conclude that $\tilde{\varphi} \in \lex^{\perp\perp}$.
\end{proof}

\begin{remark}
Dean originally used the identity $\L(E;X)^{**} \equiv \L(E;X^{**})$ to deduce the Principle of Local Reflexivity \cite{Dean73}, and that is also the approach taken in \cite{Diestel-Jarchow-Tonge, Ryan}.
Here we have gone in the opposite direction, deducing a version of Dean's identity from a version of the Principle of Local Reflexivity. This is well-known folklore in the classical case.
\end{remark}

\section{Ando-Choi-Effros liftings respecting subspaces }

In general given a bounded linear map $T : Y \to X/Z$ it is not possible to find a lifting of $T$ to $X$, i.e. a linear map $L : Y \to X$ such that $q \circ  L = T$ where $q : X \to X/Z$ is the canonical quotient map.
The classical Ando-Choi-Effros Theorem states that in the special case where $Z$ is an $M$-ideal in $X$ and $Y$ has the BAP (or is an $L_1$-predual) such a lifting does exist.
We will prove that in the case of simultaneous $M$-ideals, one can even get a simultaneous lifting.
Our approach follows closely that of \cite[Sec. II.2]{Harmand-Werner-Werner}.

Before stating the results, we need a notion that will play the role of an $L_1$-predual in the context of pairs.
$L_1$-preduals are also known as Lindenstrauss spaces, due to
Lindenstrauss' early work on them including a wealth of different characterizations: 
an excellent reference is \cite[Chap. VI]{Lindenstrauss-memoir}.
Below we introduce the corresponding concept for pairs,
using as definition 
the one property of Lindenstrauss spaces that we need for the purposes of the Ando-Choi-Effros Theorem respecting subspaces.
In a separate paper 
we prove other characterizations of such pairs in the style of Lindenstrauss' work, in particular in terms of intersection properties reminiscent of Theorem \ref{thm-geo-charac}.

\begin{definition}
A \pbs{} $\X$ is said to be \emph{$\lambda$-injective} if whenever $(F_1,F_2) \subseteq (E_1,E_2)$ are \pbss{} with $E_2 = E_1 \cap F_2$, and
$t : (F_1,F_2) \to (X_1,X_2)$ is a bounded linear map, there exists a bounded linear extension $T : (E_ 1,E_2) \to (X_1,X_2)$ with $\n{T} \le \lambda \n{t}$.
The pair $\X$ is called a \emph{Lindenstrauss pair} if $\Xss$ is 1-injective.
\end{definition}

It should be noted that Lindenstrauss pairs do exist: a trivial example is to take $X_1 = X_2 \oplus_\infty X_3$, where both $X_2$ and $X_3$ are Lindenstrauss spaces. More general examples are given by the following proposition.

\begin{proposition}
\begin{enumerate}[(i)]
\item\label{LP:individual} If $\X$ is 1-injective then both $X_1$ and $X_2$ are 1-injective, and $X_2$ is 1-complemented in $X_1$.
\item\label{LP:M-ideal-in-L} Let $X_1$ be a Lindenstrauss space, and $X_2$ an $M$-ideal in $X_1$. Then $\X$ is a Lindenstrauss pair.
\end{enumerate}
\end{proposition}

\begin{proof}
Suppose that $\X$ is 1-injective.
Let $E$ be a Banach space, $F \subset E$ a closed subspace and $t : F \to X_1$ a bounded linear map.
Applying the definition of 1-injective with $(E_1,E_2) = (E,\{0\})$ and $(F_1,F_2) = (F, \{0\})$, $t$ has an extension $T : E \to X_1$ with the same norm.
If $s : F \to X_2$ is a bounded linear operator, applying the definition with $(E_1,E_2) = (E,E)$ and $(F_1,F_2) = (F, F)$ gives a bounded linear extension $S : (E,E) \to \X$ with the same norm; note that in fact $S$ is a map from $E$ to $X_2$.
Applying the above argument to the identity map $Id_{X_2} : X_2 \to X_2$ produces a norm one projection from $X_1$ onto $X_2$.

Suppose now that $X_1$ is a Lindenstrauss space and $X_2$ is an $M$-ideal in $X_1$.
Then $X_2^\perp$ is an $L$-summand in $X_1^*$; since the latter is an $L_1$-space, so is the former \cite[Ex. 1.6(a)]{Harmand-Werner-Werner}. By considering the complementary projection, we can decompose $X_1^* \equiv X_2^\perp \oplus_1 Y$ where both $X_2^\perp$ and $Y$ are $L_1$-spaces. It follows that $X_1^{**} \equiv X_2^{\perp\perp} \oplus_\infty Y^*$. Since both $X_2^{\perp\perp}$ and $Y^*$ are 1-injective it is now clear that $\Xss$ is 1-injective and thus $\X$ is a Lindenstrauss pair.
\end{proof}

The heart of the proof of the Ando-Choi-Effros Theorem respecting subspaces is the following preparatory lemma, an adaptation of \cite[Lemma. II.2.4]{Harmand-Werner-Werner}. It deals with the fundamental step of extending a lifting defined on a finite-dimensional space to a larger finite-dimensional space. 

\begin{lemma}\label{lemma-pre-Ando-Choi-Effros}
	Suppose that $\J$ is a simultaneous $M$-ideal in $\X$, and let $q_i : X_i \to X_i/J_i$ be the quotient maps for $i=1,2$.
	Let $(F_1,F_2) \subset (E_1,E_2)$ be \pbss{} with $E_1$ finite-dimensional.
	Let $T : (E_1,E_2) \to (X_1/J_1, X_2/J_2)$ be a linear map with $\n{T} = 1$.
	If either
	\begin{enumerate}[(a)]
	\item\label{preACE:assump:projection} There exists a contractive projection $\pi : (E_1,E_2) \to (E_1,E_2)$ with $F_i = \pi(E_i)$ for $i=1,2$; or
	\item\label{preACE:assump:LS} $\J$ is a Lindenstrauss pair,
	\end{enumerate}
then, given $\eps>0$ and a contractive $L : (F_1,F_2) \to (X_1,X_2)$ such that $ q_1 \circ L = \restr{T}{F_1}$, there exists a contractive
$\tilde{L} : (E_1,E_2) \to (X_1,X_2)$ such that $q_1 \circ \tilde{L} = T$ and $\bign{ \restr{\tilde{L}}{F_1} - L} \le \eps$.
\end{lemma}

\begin{proof}
The proof is very close to that of \cite[Lemma. II.2.4]{Harmand-Werner-Werner}, but we need to carefully go through it to make sure that everything works with the extra assumption of respecting subspaces.
We start by defining
\begin{align*}
W &= \big\{ S \in \lex \mid \ran(S) \subseteq J_1 \text{ and } \ran( \restr{S}{E_2} ) \subseteq J_2 \} \equiv \Lr{E}{J}	\\
V &= \big\{ S \in W \mid  \ker(S) \supseteq F_1 \text{ and } \ker( \restr{S}{E_2} ) \supseteq F_2 \}.
\end{align*}
	
	Using Theorem \ref{thm-PLRRS-a-la-Dean}, it follows that
\begin{align*}
W^{\perp\perp} &= \big\{ S \in \lexss \mid \ran(S) \subseteq J_1^{\perp\perp} \text{ and } \ran( \restr{S}{E_2} ) \subseteq J_2^{\perp\perp} \} \\
&\equiv \Lr{E}{J^{\perp\perp}}	\\
V^{\perp\perp} &= \big\{ S \in W^{\perp\perp} \mid  \ker(S) \supseteq F_1 \text{ and } \ker( \restr{S}{E_2} ) \supseteq F_2 \}
\end{align*}	
					
	Let us now observe that $W$ is an $M$-ideal in $\lex$.
	By \cite[Lemma. VI.1.1]{Harmand-Werner-Werner}, 
	if $P : X_1^{**} \to J_1^{\perp\perp}$ is the simultaneous $M$-projection associated with $\J$, then $\tilde{P} : S \mapsto P \circ S$ is an $M$-projection on $\lexss$.
	The range of $\tilde{P}$ is obviously contained in $W^{\perp\perp}$, and it is easy to see that the range is in fact all of $W^{\perp\perp}$: take a basis for $E_2$ and complete it to a basis for $E_1$, and use this basis to write a representation of an arbitrary element of $\Lr{E}{J^{\perp\perp}}$.
	Since $W^{\perp\perp}$ is weak$^*$-closed, it follows from \cite[Cor. II.3.6]{Harmand-Werner-Werner} that $\tilde{P}$ is the adjoint of an $L$-projection and therefore $W$ is an $M$-ideal in $\lex$.
	
	Now let $L' \in \Lr{E}{X}$ be any extension of $L$ such that $q_1 \circ L' = T$; this exists because $E_1$ is finite-dimensional, and can be achieved by the same type of completing-the-basis argument as in the previous paragraph.	
	Let $B$ denote the unit ball of $\lex$. We want to prove that
	\begin{equation}\label{L-in-closure}
	L' \in \overline{ B+V },
	\end{equation}
	in order to do so we will consider $L'$ as an element of $\lexss$ and we will show that
	\begin{equation}\label{L-in-weak-star-closure}
	L' \in \overline{ B+V }^{w^*}.
	\end{equation}
	Let us first show \eqref{L-in-weak-star-closure} under assumption \eqref{preACE:assump:projection}.
	If $P$ is as above, decompose $L'$ as
	$$
	L' = \big( (Id_{X^{**}_1} - P) L' + PL' \pi \big) + PL'(Id_{E_1} - \pi).
	$$
	First note that $PL'(Id_{E_1} - \pi) \in V^{\perp\perp}$.
	Clearly $\n{PL' \pi} = \n{PL\pi} \le \n{P} \cdot\n{L} \cdot \n{\pi} \le 1$.
	Since $\ran(Id_{X^{**}_1}-P) \equiv (X_1/J_1)^{**}$, and looking at the diagram
	$$
	\xymatrix{
	E_1 \ar[dr]_{L'}\ar[r]^T &X_1/J_1 \ar[r] & (X_1/J_1)^{**}\\
			&X_1 \ar[u]^{q_1}\ar[r] &X_1^{**} \ar[u]_{Id_{X_1^{**}}-P}
	}
	$$
	it follows that $\n{(Id_{X^{**}_1} - P)L'} = \n{T} = 1$.
	Since $\tilde{P}$ is an $M$-projection, it follows that
	$$
	\n{(Id_{X^{**}_1} - P) L' + PL' \pi} = \max \big\{  \n{Id_{X_1} - P) L'},  \n{PL' \pi} \big\} = 1.
	$$
	Note also that $PL' \pi$ and $(Id_{X^{**}_1} - P) L'$ both belong to $\lexss$.
	Therefore,
	$$
	L' \in B_{\lexss} + V^{\perp\perp} = \overline{B}^{w^*} + \overline{V}^{w^*} = \overline{B+V}^{w^*}.
	$$
	If instead we assume condition \eqref{preACE:assump:LS}, from the definition of a Lindenstrauss pair there exists a contractive bounded linear map $\Lambda \in \Lr{E}{J^{\perp\perp}}$ which is a simultaneous extension for $(PL,\restr{PL}{F_2})$.
	We now decompose $L'$ as
	$$
	L' = \big( (Id_{X^{**}_1} - P) L' + \Lambda) + (PL' - \Lambda)
	$$
	and deduce \eqref{L-in-weak-star-closure} as above.
	
	Now, from \eqref{L-in-closure} there exist $R \in B$ and $S \in V$ such that $\n{ L' - (R + S) } \le \varepsilon/2$.
	Define $L'' = L' - S \in \lex$. Note that $L''$ is a simultaneous lifting for $(T,\restr{T}{E_2})$, since $S \in V \subset W$,
	but it is not guaranteed to be a contraction: we only have $\n{L''} \le (1+\varepsilon/2)$.
	We would like to perturb $L''$ slightly to obtain a map that is still a lifting but is actually a contraction.
	Now,
	\begin{align*}
	L'' \in &(L' + V) \cap (1+\varepsilon/2)B \subset (\overline{B+V}) \cap (1+\varepsilon/2)B \\
	& \subset (\overline{B+W}) \cap (1+\varepsilon/2)B \subset B + \varepsilon(B \cap W)
	\end{align*}
	where we have used \eqref{L-in-closure} in the last step of the first line, and \cite[Lemma II.2.5]{Harmand-Werner-Werner} in the last step of the second line.
	Thus there is a contraction $\tilde{L} \in \lex$ with $\| \tilde{L} - L''\| \le \varepsilon$ and $\tilde{L} - L \in W$.
	It follows that $\tilde{L}$ satisfies the desired conditions.
\end{proof}

We are now ready to prove the Ando-Choi-Effros Theorem respecting subspaces
(compare to \cite[Thm. II.2.1]{Harmand-Werner-Werner}).

\begin{theorem}\label{thm-Ando-Choi-Effros-respecting-subspaces}
	Suppose that $\J$ is a simultaneous $M$-ideal in $\X$, and let $q_i : X_i \to X_i/J_i$ be the quotient maps for $i=1,2$.
	Let $\Y$ be a \pbs{} with $Y_1$ separable, and let $T :(Y_1,Y_2) \to (X_1/J_1,X_2/J_2)$ be a linear map with $\n{T} = 1$.
	If either
	\begin{enumerate}[(a)]
	\item\label{ACE:assump:BAP} $\Y$ has the $\lambda$-BAP; or
	\item\label{ACE:assump:LS} $\J$ is a Lindenstrauss pair,
	\end{enumerate}
	then there exists $L : (Y_1,Y_2) \to (X_1,X_2)$ such that $q_1\circ L = T$.
	Moreover, $\n{L} \le \lambda$ under assumption \eqref{ACE:assump:BAP} and $\n{L} \le 1$ under assumption \eqref{ACE:assump:LS}.
\end{theorem}

\begin{proof}

	We start assuming condition \eqref{ACE:assump:LS}, since the proof is easier.
	Let $(E_n,F_n)$ be a paving for $\Y$.
	Let $E_0 = F_0 = \{0\}$ and $L_0 = 0$.
	Using Lemma \ref{lemma-pre-Ando-Choi-Effros}, we can inductively define a sequence of contractions $L_n : E_n \to X_1$ such that
	$$
	q_1 \circ L_n = \restr{T}{E_n}, \qquad L_n(F_n) \subset X_2, \qquad \text{ and }  \n{ \restr{L_n}{E_{n-1}} - L_{n-1}} \le 2^{-n}.
	$$
	For any $y \in \bigcup_n E_n$, the sequence $(L_ny)$ is eventually defined and Cauchy. Hence $Ly := \lim_{n\to\infty} L_n$ defines a contraction on  $\bigcup_n E_n$ that can be extended to a contraction $L : Y_1 \to X_1$ that clearly has the desired properties.

	Now assume condition \eqref{ACE:assump:BAP}.
	Consider the following diagram induced on the corresponding spaces of convergent sequences
	$$
	\xymatrix{
	c(Y_1) \ar[r]^{c(T)} & c(X_1/J_1)  & c(X_1) \ar[l]_{c(q_1)} & c(J_1) \ar[l] \\ 
	c(Y_2) \ar[u] \ar[r]_{c(T|_{Y_2})} & c(X_2/J_2) \ar[u]  & c(X_2) \ar[u] \ar[l]^{c(q_2)} & c(J_2) \ar[u] \ar[l] \\ 
	}
	$$
	Just as in the proof of \cite[Thm. II.2.1]{Harmand-Werner-Werner} but additionally using Theorem \ref{thm-geo-charac}, note that $c(T)$ is a contraction, and for $i=1,2$ we have that $c(q_i) : c(X_i) \to c(X_i/J_i)$ is a quotient map with kernel $c(J_i)$, so $c(X_i/J_i) \cong c(X_i)/c(J_i)$ and moreover 
	 $\big( c(J_1), c(J_2) \big)$ is a simultaneous $M$-ideal in $\big( c(X_1), c(X_2))$.

	Since $Y_1$ is separable and $\Y$ has the $\lambda$-BAP, by standard arguments there exists a sequence of finite-rank operators $S_n : Y_1 \to Y_1$ with $\n{S_n} \le \lambda$ converging strongly to $Id_{Y_1}$ and leaving $Y_2$ invariant.
	For $i=1,2$ we define an auxiliary subspace $H_i \subset c(Y_i)$ as the closed linear span of the sequences
	$$
	(S_1y, S_2y, \dotsc, S_{m-1}y, S_my, S_my, \dotsc)
	$$
	where $m \in \N$ and $y \in Y_i$.
	Note that $(S_ny)_n \in H_i$ for every $y \in X_i$.
	For $m \in \N$ and $(y_n)_n \in H_i$ we define 
	$$
	\pi^i_m\big( (y_n)_n \big) = (y_1, y_2, \dotsc, y_{m-1}, y_m, y_m, \dotsc).
	$$
	Note that the $\pi^i_m$ form an increasing sequence of contractive finite-rank projections on $H_i$ converging strongly to $Id_{H_i}$, and each $\pi^2_m$ is simply the restriction of $\pi_m^1$ to $H_2$.
	If we let $E^i_m = \ran(\pi_m^i)$, we get finite-dimensional subspaces $E^2_m \subset E^1_m$ with $E^i_m \subset Y_i$ for $i = 1,2$; Define also $E^1_0 = \{0\}$ and $L_0 = 0$.
	Using Lemma \ref{lemma-pre-Ando-Choi-Effros}, we define inductively a sequence of contractions $L_m : E^1_m \to c(X_1)$ such that
	$$
	c(q_1) \circ L_m = \restr{c(T)}{E^1_m}, \quad L_m(E^2_m) \subset c(X_2), \quad \n{ \restr{L_m}{E^1_{m-1}} - L_{m-1} } \le 2^{-m}.
	$$
	For $y \in \bigcup_m E_m$ the sequence is eventually defined and Cauchy, hence we obtain a contractive linear map $\Lambda : H_1 \to c(X_1)$ such that
	$$
	c(q_1) \Lambda = c(T)|_{H_1} \quad \text{and} \quad \Lambda(H_2) \subset c(X_2).  
	$$
	We now define $L : Y_1 \to X_1$ by 
	$$
	Ly = \limit \Lambda \big( (S_my)_{m} \big)
	$$
	Note that $L(Y_2) \subset X_2$ and moreover for any $y \in Y_1$
	$$
	\n{Ly} \le \n{\Lambda} \sup_m \n{S_m y} \le \lambda \n{y},
	$$
	and
	\begin{align*}
		q_1(Ly) &= \limit c(q_1) \big( \Lambda \big( (S_my)_{m} \big) \big) \\
		&= \limit (c(q_1) \Lambda) \big( (S_my)_{m} \big) \\
		&= \lim_m T(S_my) = Ty. 
	\end{align*}
\end{proof}

As a first consequence, we get a version of the Michael-Pe{\l}czy{\'n}ski extension theorem \cite{Michael-Pelczynski} respecting subspaces.

\begin{corollary}\label{cor-Michael-Pelczynski}
Let $K$ be a compact Hausdorff space.
Suppose $X_2 \subseteq X_1$ are closed subspaces of $C(K)$, and $D \subseteq K$ is closed.
If $(\restr{X_1}{D}, X_1)$ and $(\restr{X_2}{D}, X_2)$ both have the bounded extension property,
$\restr{X_1}{D}$ is separable and the pair $(\restr{X_1}{D} ,\restr{X_2}{D} )$ has the $\lambda$-BAP,  
then there is a linear extension operator $T : (\restr{X_1}{D}, \restr{X_2}{D} ) \to \X$  (that is, $(Tf)(k) = f(k)$ for any $f \in \restr{X_1}{D}$ and $k \in D$) 
with $\n{T} \le \lambda$.
\end{corollary}

\begin{proof}
This follows immediately from Corollary \ref{cor-joint-bounded-extension-property} and Theorem \ref{thm-Ando-Choi-Effros-respecting-subspaces}.
\end{proof}

\begin{remark}
Let us mention some known situations where the hypotheses of Corollary \ref{cor-Michael-Pelczynski} are satisfied:
\begin{enumerate}[(a)]
\item If $X \subset C(K)$ is a subalgebra and $D$ is a $p$-set for $X$ (that is, the intersection of finitely many sets of the form $f^{-1}(\{1\})$ with $f \in B_X$), then the pair $(\restr{X}{D},X)$ has the bounded extension property \cite[p. 15]{Harmand-Werner-Werner}.
Therefore, if $X_2 \subset X_1 \subset C(K)$ are subalgebras and $D$ is a $p$-set for $X_2$, then both $(\restr{X_1}{D}, X_1)$ and $(\restr{X_2}{D}, X_2)$ have the bounded extension property.
\item In regards to the hypothesis of $(\restr{X_1}{D} ,\restr{X_2}{D} )$ having BAP, it must be said that
not too many nontrivial (i.e. where the subspace is not complemented) examples of pairs $(Y_1,Y_2)$ with the BAP are known:
when $\dim (Y_1/Y_2) < \infty$ and $Y_2$ has BAP \cite[Prop. 1.8]{Figiel-Johnson-Pelczynski};
when $Y_1$ is an $\mathcal{L}_\infty$ space and $Y_2$ has BAP  \cite[Thm. 2.1]{Figiel-Johnson-Pelczynski}.
In particular, when $Y_1 = C(D)$ and $Y_2$ has BAP.
It should be noted that even when $Y_2$ is known to have MAP, the aforementioned results from \cite{Figiel-Johnson-Pelczynski} do not give MAP for the pair $(Y_1,Y_2)$.
In the context of extensions it is particularly interesting to obtain extension operators with norm one, so it would be desirable to find examples of pairs of the form $(\restr{X_1}{D} ,\restr{X_2}{D} )$ having MAP.
\end{enumerate}
\end{remark}


We will now apply our Ando-Choi-Effros lifting theorem respecting subspaces to get a result promised in the introduction: a characterization of the BAP for pairs in terms of the existence of a simultaneous lifting for the associated Lusky-inspired diagram \eqref{full-diagram}.

\begin{theorem}\label{thm-BAP-equivalent-to-simultaneous-splitting}
Let $X_1$ be a separable Banach space, $X_2$ a closed subspace of $X_1$, $(E_n,F_n)$ a paving for $\X$, and let $\lambda \ge 1$.
The following are equivalent:
\begin{enumerate}[(i)]
	\item\label{BAP-equiv-split:BAP} $\X$ has $\lambda$-BAP.
	\item\label{BAP-equiv-split:split} 
	The pair of short exact sequences
	\begin{equation}\label{eqn-Lusky-pair}
	\xymatrix{
	0 \ar[r] & c_0(E_n)\ar[r]^{j_1} & c(E_n)\ar[r]^{q_1} & X_1 \ar[r]  &0\\
	0 \ar[r] & c_0(F_n)\ar[r]^{j_2}\ar[u] & c(F_n)\ar[r]^{q_2}\ar[u] & X_2 \ar[r]\ar[u]  &0 
	}
	\end{equation}
	admits a simultaneous (linear) lifting of norm less than or equal to $\lambda$.
\end{enumerate}
\end{theorem}	

\begin{proof}
	(ii) $\Rightarrow$ (i):
	Let $R : X_1 \to c(E_n)$ be a simultaneous linear lifting with $\n{R} \le \lambda$.
Taking the compositions of $R$ followed by the projection on the $n$-th coordinate gives a sequence of finite-rank maps on $X_1$ leaving $X_2$ invariant, with norms uniformly bounded by $\lambda$ and converging pointwise to the identity; standard arguments yield the BAP for the pair $\X$ (see, for example, \cite[Prop. 4.3]{Ryan}).
	
	(i) $\Rightarrow$ (ii):	
	Consider the map $Id_{X_1} : (X_1, X_2) \to (X_1,X_2)$, and observe that $c(E_n)/c_0(E_n) \equiv X_1$ and $c(E_n)/c_0(E_n) \equiv X_1$.
	By Corollary \ref{cor-example-simultaneous-c0} and Theorem \ref{thm-Ando-Choi-Effros-respecting-subspaces}, there exists a linear map $L : \X \to \big( c(E_n), c(F_n) \big)$ of norm at most $\lambda$ such that $q_1 \circ L = Id_{X_1}$.		
\end{proof}

Note that as a consequence, if condition \eqref{BAP-equiv-split:split} of Theorem \ref{thm-BAP-equivalent-to-simultaneous-splitting} is satisfied for one paving of $\X$ then it is satisfied for any paving of $\X$.


\section{Lipschitz BAP for pairs}

Let $\lambda \ge 1$. Recall that the Banach space $X$ is said to have the \emph{$\lambda$-Lipschitz bounded approximation property} ($\lambda$-Lipschitz BAP) \cite[Defn. 5.2]{Godefroy-Kalton-03} if, for each $\varepsilon>0$ and compact set $K \subset X$, there exists a Lipschitz map $S : X \to X$ with finite-dimensional range and such that $\Lip(S) \le \lambda$ and $\n{S(x)-x} \le \varepsilon$ for each $x \in K$. We now define the corresponding Lipschitz version of the BAP for pairs.

\begin{definition}
Let $(X,Y)$ be a \pbs.
We say that the pair $(X,Y)$ has the  \emph{$\lambda$-Lipschitz BAP} if, for each $\varepsilon>0$ and compact set $K \subset X$, there exists a Lipschitz map $S : (X,Y) \to (X,Y)$ with finite-dimensional range, $\Lip(S) \le \lambda$, and $\n{S(x)-x} \le \varepsilon$ for each $x \in K$.
\end{definition}

The celebrated Godefroy-Kalton theorem \cite[Thm. 5.3]{Godefroy-Kalton-03} states that for an individual Banach space $X$, the BAP and the Lipschitz BAP are equivalent. 
At this point, it is natural to wonder whether the analogous equivalence holds for the case of pairs.
One of the implications is trivial,
since the BAP for $(X,Y)$ obviously implies the Lipschitz BAP for $(X,Y)$ (and with the same constant). 
We show in Theorem \ref{thm-Lip-BAP-GK} below that the equivalence does hold in the presence of an additional hypothesis, a version for pairs of the Lipschitz-lifting property \cite[Defn. 5.2]{Godefroy-Kalton-03}.
We also show, with an example due to W.B. Johnson, that the equivalence does not hold in general.

\begin{definition}
The \pbs{}  $(X,Y)$ is said to have the (isometric) \emph{Lipschitz-lifting property} if there exists a (norm one) continuous linear map $T : (X,Y) \to \big( \fbs(X), \fbs(Y) \big)$ such that $\beta_X \circ T = Id_X$.
\end{definition}

The following is a version for pairs of  \cite[Prop. 2.6]{Godefroy-Kalton-03}. In particular, it implies that if a right-isometric pair of short exact sequences admits a simultaneous Lipschitz lifting, then the bidual pair of short exact sequences admits a simultaneous linear lifting. As a consequence, if $X$ is a reflexive space, then the BAP for $(X,Y)$ is equivalent to the Lipschitz BAP for $(X,Y)$.

\begin{proposition}\label{prop-splitting-of-bidual}
Suppose that the right-isometric pair of short exact sequences 
$$
\xymatrix{
	0 \ar[r] & Z_1 \ar[r] & Y_1 \ar[r]^{B_1} &X_1 \ar[r]  & 0\\
	0 \ar[r] & Z_2 \ar[r]\ar[u] & Y_2 \ar[r]^{B_2} \ar@{^{(}->}[u] &X_2 \ar[r] \ar@{^{(}->}[u]  & 0\\
	}
$$
admits a simultaneous Lipschitz lifting $L : (X_1,X_2) \to  (Y_1,Y_2)$.
Then for the dual pair of sequences
$$
\xymatrix{
	0 \ar[r] & X^*_1 \ar[r]^{B_1^*}\ar[d]^{r} & Y^*_1 \ar[r]\ar[d]^{r'} &Z^*_1 \ar[r] \ar[d]  & 0\\
	0 \ar[r] & X^*_2 \ar[r]^{B_2^*} & Y^*_2 \ar[r] &Z^*_2 \ar[r]  & 0\\
	}
$$
where $r$ and $r'$ are the restriction maps, there exist linear maps $T_1 : Y_1^* \to X_1^*$ and $T_2 : Y^*_2 \to X_2^*$ such that $T_1 \circ B_1^* = Id_{X_1^*}$, $T_2 \circ B_2^* = Id_{X_2^*}$ and $r \circ T_1 = T_2 \circ r'$.
\end{proposition}

\begin{proof}
By \cite[Prop. 7.5]{Benyamini-Lindenstrauss}, there are contractive linear surjective projections $P_1 : X_1^\# \to X_1^*$ and $P_2 : X_2^\# \to X_2^*$ such that $rP_1 = P_2 \tilde{r}$,  where $\tilde{r} :X_1^\# \to X_2^\#$ is the restriction operator.
Let $L_1^\# : Y_1^\# \to X_1^\#$ and $L_2^\# : Y_2^\# \to X_2^\#$ be given by $f \mapsto f \circ L$ and $g \mapsto g \circ \restr{L}{X_2}$, respectively.
Choosing $T_1 = P_1 \circ \restr{L_1^\#}{Y_1^*}$ and $T_2 = P_2 \circ \restr{L_2^\#}{Y_2^*}$ gives the desired maps. 
\end{proof}

Next, a characterization of the Lipschitz-lifting property in terms of the existence of simultaneous liftings. Compare to \cite[Prop. 2.8]{Godefroy-Kalton-03}.

\begin{proposition}\label{prop-lifting-property}
Let $\X$ be a \pbs. Then $\X$ has the Lipschitz-lifting property if and only if every right-isometric
pair of short exact sequences
$$
\xymatrix{
	0 \ar[r] & Z_1 \ar[r] & Y_1 \ar[r]^{q_1} &X_1 \ar[r]  & 0\\
	0 \ar[r] & Z_2 \ar[r]\ar[u] & Y_2 \ar[r]^{q_2} \ar@{^{(}->}[u] &X_2 \ar[r] \ar@{^{(}->}[u]  & 0\\
	}
$$
which admits a simultaneous Lipschitz lifting also admits a simultaneous linear lifting.
\end{proposition}

\begin{proof}
Let $L : \X \to \Y$ be a simultaneous Lipschitz lifting for $(q_1,q_2)$, and let $L_2 = \restr{L_1}{X_2}$.
By the proof of \cite[Prop. 2.8]{Godefroy-Kalton-03} we get a commutative diagram
$$
\xymatrix{
	0 \ar[r] & Z_{X_1} \ar[r] \ar[d]^{V_1} & \fbs(X_1) \ar[d]^{\overline{L_1}} \ar[r]^{\beta_{X_1}} &X_1 \ar@{=}[d] \ar[r]  & 0\\
	0 \ar[r] & Z_1 \ar[r] & Y_1 \ar[r]^{q_1} &X_1 \ar[r]  & 0\\
	0 \ar[r] & Z_2 \ar[r]\ar[u] & Y_2 \ar[r]^{q_2}\ar@{^{(}->}[u] &X_2 \ar[r] \ar@{^{(}->}[u]  & 0\\
	0 \ar[r] & Z_{X_2} \ar[r] \ar[u]_{V_2} & \fbs(X_2) \ar[u]^{\overline{L_2}} \ar[r]_{\beta_{X_2}} &X_2 \ar@{=}[u] \ar[r]  & 0\\
	}
$$
If $T : X_1 \to \fbs(X_1)$ is a bounded linear map such that $\beta_{X_1} T = Id_{X_1}$ and $T(X_2) \subset \fbs(X_2)$, it is clear that $\overline{L_1} \circ T$ is a simultaneous linear lifting for $(q_1,q_2)$.
\end{proof}

The same argument gives the following isometric version, as in \cite[Prop. 2.9]{Godefroy-Kalton-03}.

\begin{proposition}
Let $\X$ be a \pbs{} with the isometric Lipschitz-lifting property, and consider the following
right-isometric pair of  isometric short exact sequences
$$
\xymatrix{
	0 \ar[r] & Z_1 \ar[r] & Y_1 \ar[r]^{q_1} &X_1 \ar[r]  & 0\\
	0 \ar[r] & Z_2 \ar[r]\ar[u] & Y_2 \ar[r]^{q_2} \ar@{^{(}->}[u] &X_2 \ar[r] \ar@{^{(}->}[u] & 0\\
	}
$$
If there exists an isometry $L :\X \to \Y$ (not necessarily linear) such that  $q_1 \circ L = Id_{X_1}$, then there is also a linear isometry $V : \X \to \Y$ such that $q_1\circ V = Id_{X_1}$.
\end{proposition}

A basic example of pairs with the Lipschitz-lifting property is given by the following lemma, corresponding to \cite[Lemma 2.10]{Godefroy-Kalton-03}.

\begin{lemma}\label{lemma-fbs-has-Lipschitz-lifting-property}
Let $X_1 \subseteq X_2$ be metric spaces. Then the pair $\big( \fbs(X_1), \fbs(X_2) \big)$ has the isometric Lipschitz lifting property.
\end{lemma}

\begin{proof}
As in the proof of \cite[Lemma 2.10]{Godefroy-Kalton-03},
the isometry $t : \delta_{\fbs(X_1)} \delta_{X_1} : X_1 \to \fbs(\fbs(X_1)) = \fbs^2(X_1)$
induces, by the universal property of the free space,
a linear map $T : \fbs(X_1) \to \fbs^2(X_1)$ with $\n{T} = 1$, $T\delta_{X_1}(x) = \delta_{\fbs(X_1)}\delta_{X_1}(x)$ for every $x \in X_1$.
Note that for every $x \in X_2$ we have $t(x) \in \fbs^2(X_2)$, so for any $m \in \fbs(X_2)$ we have $T(m) \in \fbs^2(X_2)$; that is, $T(\fbs(X_2)) \subset \fbs^2(X_2)$.
Since $T\delta_{X_1}(x) = \delta_{\fbs(X_1)}\delta_{X_1}(x)$ for every $x \in X_1$, it follows that $\beta_{\fbs(X_1)}T\delta_{X_1}(x) = \delta_{X_1}(x)$, from where we conclude $\beta_{\fbs(X_1)}T = Id_{\fbs(X_1)}$.
\end{proof}

\begin{remark}
It would be desirable to find conditions, perhaps separability as in \cite[Thm. 3.1]{Godefroy-Kalton-03} together with some extra hypotheses, to find examples of other pairs with the Lipschitz-lifting property.
\end{remark}

The following is a characterization of the pairs $(X,Y)$ having the Lipschitz-lifting property, in the style of
\cite[Thm. 2.12]{Godefroy-Kalton-03}.

\begin{proposition}
Let $(X,Y)$ be a \pbs. Define $\mathcal{G}(X) = \ker(\beta_X) \oplus X$ and $\mathcal{G}(Y) = \ker(\beta_Y) \oplus Y$.
Then $\big(\fbs(X), \fbs(Y)\big)$ is simultaneously Lipschitz isomorphic to $\big( \mathcal{G}(X), \mathcal{G}(Y) \big)$.
Moreover, these two pairs are simultaneously linearly isomorphic if and only if $(X,Y)$ has the lifting property.
\end{proposition}

\begin{proof}
The map $L : \fbs(X) \to \mathcal{G}(X)$ defined by $
\mu \mapsto ( \mu - \delta_X\beta_X(\mu), \beta_X(\mu) )$ is a simultaneous Lipschitz isomorphism between $\big(\fbs(X), \fbs(Y)\big)$ and $\big( \mathcal{G}(X), \mathcal{G}(Y) \big)$.
If $T : X \to \fbs(X)$ is a simultaneous linear lifting for $(X,Y)$, replacing $\delta_X$ by $T$ above shows that the pairs are simultaneously linearly isomorphic.
Conversely, if  $\big(\fbs(X), \fbs(Y)\big)$ and $\big( \mathcal{G}(X), \mathcal{G}(Y) \big)$ are simultaneously linearly isomorphic, Lemma \ref{lemma-fbs-has-Lipschitz-lifting-property} and a bit of diagram chasing (as in \cite[Lemma 2.11]{Godefroy-Kalton-03}) gives the desired result.
\end{proof}

Now we show the equivalence between the BAP and the Lipschitz BAP for pairs, when the Lipschitz-lifting property is present.
The proof follows along the lines of \cite[Cor. 2.3]{BorelMathurin}.

\begin{theorem}\label{thm-Lip-BAP-GK}
Let $X$ be a separable Banach space, and $Y$ a closed subspace of $X$. 
Suppose that $(X,Y)$ has the Lipschitz-lifting property.
Then the following are equivalent:
\begin{enumerate}[(i)]
	\item\label{thm-Lip-BAP:BAP} $(X,Y)$ has the BAP.
	\item\label{thm-Lip-BAP:LipBAP} $(X,Y)$ has the Lipschitz BAP.
\end{enumerate}
\end{theorem}	

\begin{proof}
We have already observed that the implication
\eqref{thm-Lip-BAP:BAP} $\Rightarrow$ \eqref{thm-Lip-BAP:LipBAP} trivially holds in general.

Suppose now that $(X,Y)$ has the Lipschitz BAP.
By standard arguments (say, as in the proof of \cite[Thm. 2.2]{BorelMathurin}) we can construct a sequence of Lipschitz maps 
$\sigma_n : X \to X$ with finite-dimensional range so that $\sigma_n(Y) \subset Y$ for all $n\in\N$, $\sup_n \Lip(\sigma_n) < \infty$ and $\lim_n \sigma_nx = x$ for all $x \in X$.
Choosing pairs of finite-dimensional spaces $(E_n,F_n) \subset (X,Y)$ with $\sigma_n(X) \subset E_n$, $\sigma_n(Y) \subset F_n$, $\bigcup_n E_n$ dense in $X$ and $\bigcup_n F_n$ dense in $Y$, the map $T : X \to c(E_n)$ given by $Tx = (\sigma_n(x))_n$ is a simultaneous Lipschitz lifting for \eqref{eqn-Lusky-pair}.
Since $(X,Y)$ has the Lipschitz-lifting property, by Proposition \ref{prop-lifting-property}, there exists a simultaneous linear lifting $L : X \to c(E_n)$. By Theorem \ref{thm-BAP-equivalent-to-simultaneous-splitting}, we conclude $(X,Y)$ has BAP.
\end{proof}

In view of the preceding result, one would want to know whether all pairs of separable Banach spaces enjoy the Lipschitz-lifting property.
Unfortunately, that is not the case. 
In the rest of this section, we show a way to find pairs of separable Banach spaces without the Lipschitz-lifting property.
The argument is indirect, and relies on Theorem \ref{thm-Lip-BAP-GK};
we are indebted to Prof. W.B. Johnson for showing it to us.
We start by showing that the Lipschitz BAP for $X$ and $Y$ individually, together with a Lipschitz retraction from $X$ onto $Y$, imply the Lipschitz BAP for the pair $(X,Y)$.

\begin{proposition}\label{prop-retraction-implies-LipBAP}
Let $(X,Y)$ be a \pbs{}. If $X$ has $\lambda$-Lipschitz BAP, $Y$ has $\mu$-Lipschitz BAP and there is a Lipschitz retraction $P$ from $X$ onto $Y$, then the pair $(X,Y)$ has the $C$-Lipschitz BAP with $C = \mu\Lip(P) + \lambda(1 + \Lip(P))$.
\end{proposition}

\begin{proof}
Let $\varepsilon>0$, and let $K \subset X$ be a compact set. Note that $P(X) \subset Y$ and $(Id_X-P)(K) \subset X$ are also compact.
Thus, there exist Lipschitz maps with finite-dimensional range $T : Y \to Y$ and $S : X \to X$ such that
$\Lip(T) \le \mu$, $\Lip(S) \le \lambda$, and for every $x \in K$, $\n{ TP(x) - P(x) } \le \varepsilon/2$ and $\n{S(x-P(x)) - (x-P(x))} \le \varepsilon/2$. Note that in addition we may assume $S(0)=0$.
Now consider the map $R = TP + S(Id_X-P) : X \to X$. Clearly $R$ has finite-dimensional range, it has Lipschitz constant at most
$\Lip(T) \Lip(P) + \Lip(S) \Lip(Id_X - P) \le \mu\Lip(P) + \lambda(1 + \Lip(P))$,
for any $y \in Y$ we have $R(y) = T(y) + S(0) = T(y) \in Y$, and for every $x \in K$
\begin{multline*}
\n{R(x) -x} = \n{ TP(x) +S(x-P(x))  - x} \\
\le \n{TP(x) - P(x)} + \n{ S(x-P(x)) - (x-P(x)) } < \varepsilon.
\end{multline*}
\end{proof}

We are now ready for the example.

\begin{proposition}
There exists a separable \pbs{} $(X,Y)$ with the Lipschitz BAP, but without the BAP.
In particular, $(X,Y)$ does not have the Lipschitz-lifting property.
\end{proposition}

\begin{proof}
Let $Z$ be a separable Banach lattice without the Approximation Property \cite{Szankowski},
and let $(E_n)$ be a paving of $Z$. 
Let $X = c(E_n)$ and $Y = c_0(E_n)$;
clearly $X$ and $Y$ are separable and have the BAP. 
In particular, they both have the Lipschitz BAP.
By \cite[Thm. 5.2]{Kalton-Uniform-Structure}, there is a Lipschitz retraction
from $Y^{**} = \ell_\infty(E_n)$ onto $Y$.
Restricting this map to $X = c(E_n)$ gives a Lipschitz retraction from $X$ onto $Y$, and now
it follows from Proposition \ref{prop-retraction-implies-LipBAP} that
the pair $(X,Y)$ has the Lipschitz BAP.
If the pair $(X,Y)$ had the BAP then $X/Y$ would have the BAP as well by \cite[Cor. 1.2]{Figiel-Johnson-Pelczynski}, but $X/Y$ is isometric to $Z$.
The last part of the conclusion now follows from Theorem \ref{thm-Lip-BAP-GK}.
\end{proof}


\section{BAP for pairs of Lipschitz-free spaces}

The Godefroy-Kalton theorem \cite[Thm. 5.3]{Godefroy-Kalton-03} not only shows the equivalence between the BAP and the Lipschitz BAP for a Banach space $X$, but also that these properties are equivalent to the BAP for the corresponding Lipschitz-free space $\fbs(X)$. We do not know whether a similar result holds for the BAP for pairs.
Nevertheless, our Ando-Choi-Effros theorem respecting subspaces can be used to characterize the BAP for pairs of free spaces over compact metric spaces.
The result is a version for pairs of \cite[Thm. 2.1]{Godefroy-Extensions-and-BAP}.
Before stating it, let us introduce some notation.
If $K \subset M$ are metric spaces, we will always assume that they share the same distinguished point whenever we consider their associated Lipschitz-free spaces or the spaces of Lipschitz functions defined on them.
If $(X,Y)$ is a \pbs{}, we denote
$$
\Lip_0(M,K; X,Y) = \big\{ f \in \Lip_0(M,X) \st f(K) \subset Y \big\}
$$
with the norm inherited from $\Lip_0(M;X)$.
If $K$ and $M$ are compact metric spaces and $T :\Lip_0(M,K; X,Y) \to \Lip_0(M,K; X,Y)$ is a bounded linear operator, we denote by $\n{T}_L$ its norm when both the domain and the codomain are equipped with the Lipschitz norm, and by $\n{T}_{L,\infty}$ its norm when the domain space is equipped with the Lipschitz norm and the range space with the uniform norm.
A subset $S$ of a metric space $M$ is said to be $\varepsilon$-dense if for all $x \in M$, $\inf\{ d(x,s) \st s \in S \} \le \varepsilon$.
 
\begin{theorem}\label{Thm-ext}
Let $K \subset M$ be compact metric spaces.
Let $(K_n)_n$ and $(M_n)_n$ be sequences of finite $\varepsilon_n$-dense subsets of $K$, respectively $M$, containing the distinguished point and with $K_n \subseteq M_n$ and $\lim \varepsilon_n = 0$.
For a function $f$ defined on $M$ (resp. $K$), we denote $R_n(f)$ its restriction to $M_n$.
 The following are equivalent:
 \begin{enumerate}[(i)]
 \item\label{Thm-ext:BAP}
 The pair $\big( \fbs(M), \fbs(K) \big)$ has the $\lambda$-BAP.
 \item\label{Thm-ext:alm-ext-X}
 There exist $\alpha_n \ge 0$ with $\lim \alpha_n = 0$ such that for every \pbs{} $(X,Y)$, there exist linear operators
 $U_n : \Lip_0(M_n,K_n; X,Y) \to \Lip_0(M,K; X,Y)$ 
such that $\n{U_n}_L \le \lambda$ 
 and
 $\n{R_nU_n - Id_{\Lip_0(M,K; X,Y)}}_{L,\infty} \le \alpha_n$.
 \item\label{Thm-ext:alm-ext-R}
There exist linear operators $U_n : \Lip_0(M_n) \to \Lip_0(M)$ such that
$\n{U_n}_L \le \lambda$, the sequence
 $\n{R_n U_n - Id_{\Lip_0(M_n)}}_{L,\infty}$ converges to $0$ and $U_n( J_{K_n} \cap \Lip_0(M_n) ) \subset J_K \cap  \Lip_0(M)$.
 \item \label{Thm-ext:ext-X}
For every \pbs{} $(X,Y)$, there exist $\beta_n \ge 0$ with $\lim \beta_n = 0$ such that for every $1$-Lipschitz function $F : (M_n,K_n) \to (X,Y)$, there exists a $\lambda$-Lipschitz function $H : (M,K) \to (X,Y)$ such that 
 $\n{R_n(H) - F}_{\ell_\infty(M_n,X)} \le \beta_n$.
 \end{enumerate}
 \end{theorem}
 
 \begin{proof}
\eqref{Thm-ext:BAP} $\Rightarrow$ \eqref{Thm-ext:alm-ext-X}
Note that $\big(\fbs(K_n)\big)_n$ and $\big(\fbs(M_n)\big)_n$ are increasing sequences of finite-dimensional subspaces of $\fbs(K)$, resp. $\fbs(M)$, with the former having dense union in $\fbs(K)$ and the latter in $\fbs(M)$. Moreover, $\fbs(K_n) \subseteq \fbs(M_n)$ for each $n\in\N$.
By Theorem \ref{thm-BAP-equivalent-to-simultaneous-splitting},
there exists a simultaneous linear lifting $L : \big( \fbs(M), \fbs(K) \big) \to \big( c\big(\fbs(M_n)\big), c\big(\fbs(K_n)\big) \big)$ with $\n{L} \le \lambda$.
Let $\pi_n :c\big(\fbs(M_n)\big) \to \fbs(M_n)$ be the canonical projection.
Define $g_n := \pi_n \circ L \circ \delta_M : M \to \fbs(M_n)$.
Note that $g_n(K) \subset \fbs(K_n)$.
The maps $g_n$ are clearly $\lambda$-Lipschitz, and for every $x \in M$ we have 
$\lim \n{g_n(x) - \delta_M(x)} = 0$.
Since $M$ is compact, this implies by an equicontinuity argument that if we let
$$
\alpha_n = \sup_{x\in M} \n{ g_n(x) - \delta_M(x) }_{\fbs(M)}
$$
then $\lim \alpha_n = 0$.
Let now $X$ be a Banach space, and $F : (M_n,K_n) \to (X,Y)$ be a Lipschitz map. By the universal property of the free space, there exists a unique bounded linear map $\overline{F} : \fbs(M_n) \to X$ such that $\overline{F} \circ \delta_{M_n} = F$ and $\n{\overline{F}} = \Lip(F)$. Note that $\overline{F}$ depends linearly on $F$. We now define $U_nF = \overline{F} \circ g_n$.
Since $g_n(K) \subset \fbs(K_n)$, it follows that
$(U_nF)(K) = \overline{F}( g_n(K) ) \subset \overline{F}( \fbs(K_n) ) \subset Y$, where the last equality follows from the fact that $F(K_n) \subset Y$.
Thus $U_n$ defines a map from $\Lip_0(M_n,K_n; X,Y)$ to $\Lip_0(M,K; X,Y)$,
and it is easy to see that the sequence $U_n$ satisfies the requirements of \eqref{Thm-ext:alm-ext-X}.

 \item\eqref{Thm-ext:BAP} $\Rightarrow$ \eqref{Thm-ext:alm-ext-R}
Construct the operators $U_n$ as in the proof of the previous implication, with $X = Y = \R$.
Now let $F \in  J_{K_n} \cap \Lip_0(M_n)$. It follows that $\overline{F}$ is identically zero on $\fbs(K_n)$, and thus $U_nF$ vanishes on $K$ (since $U_nF = \overline{F} \circ g_n$ and $g_n(K) \subset \fbs(K_n)$).

\eqref{Thm-ext:alm-ext-X} $\Rightarrow$ \eqref{Thm-ext:ext-X}
It suffices to take $H = U_n(F)$ and $\beta_n = \alpha_n$.

\eqref{Thm-ext:alm-ext-R} $\Rightarrow$ \eqref{Thm-ext:BAP}
Let $\gamma_n = \n{R_n U_n - Id_{\Lip_0(M_n)}}_{L,\infty} $, so that $\lim \gamma_n = 0$.
If $H \in \Lip_0(M)$, then
$$
\n{ R_n\big[ U_nR_n(H) -H \big] }_{\ell_\infty(M_n)} = \n{ R_nU_nR_n(H) -R_n(H) }_{\ell_\infty(M_n)} \le \gamma_n \n{H}_L
$$
Now let $T_n = U_n \circ R_n : \Lip_0(M) \to \Lip_0(M)$. Note that $\n{T_n}_L \le \lambda$, and since $M_n$ is $\varepsilon_n$-dense in $M$ with $\lim \varepsilon_n = 0$, it follows that for every $H \in \Lip(M)$ one has
$$
\lim \n{T_n(H) - H}_{\ell_\infty(M)} = 0.
$$
The operator $R_n$ is a finite rank operator which is weak$^*$-to-norm continuous, hence so is $T_n = U_n \circ R_n$. In particular, there exists $A_n : \fbs(M) \to \fbs(M)$ such that $A_n^* = T_n$.
It is clear that $\n{A_n}_{\fbs(M)} \le \lambda$, and that the sequence $(A_n)$ converges to the identity for the weak operator topology.
Moreover, let $\mu \in \fbs(K)$ and $f \in \fbs(K)^\perp = J_K \cap \Lip_0(M)$.
Then
$$
\pair{f}{A_n\mu} = \pair{T_n f}{\mu} = \pair{U_nR_nf}{ \mu} = 0,
$$
since $R_n(f) \in J_{K_n} \cap \Lip_0(M_n)$ and $U_n( J_{K_n} \cap \Lip_0(M_n) ) \subset J_K \cap \Lip_0(M)$.
Therefore, $A_n(\fbs(K)) \subset \fbs(K)$.
This shows \eqref{Thm-ext:BAP}.

\eqref{Thm-ext:ext-X} $\Rightarrow$ \eqref{Thm-ext:BAP}
Consider the space $X = \ell_\infty( \fbs(M_n) )$, its closed subspace $Y =  \ell_\infty( \fbs(K_n) )$ and the maps $F_n := j_n \circ \delta_{M_n} :  (M_n, K_n) \to (X,Y)$ where $j_n : \fbs(M_n) \to X$ is the obvious injection ($(j_n(\mu))_k = 0$ for $k \not= n$ and $(j_n(\mu))_n = \mu$).
Note that each $F_n$ is an isometric injection.
By \eqref{Thm-ext:ext-X}, there exist $\lambda$-Lipschitz maps $H_n : (M,K) \to (X,Y)$ such that
 $\n{R_n(H_n) - F_n}_{\ell_\infty(M_n,X)} \le \beta_n$.
 Let $S_n := P_n \circ H_n : (M,K) \to \big( \fbs(M_n), \fbs(K_n) \big)$, where $P_n : (X,Y) \to \big( \fbs(M_n), \fbs(K_n) \big)$ is the canonical projection onto the $n$-th coordinate.
 Note that each $S_n$ is a Lipschitz map with $\Lip(S_n) \le \lambda$, and thus it has a linear extension
 $\overline{S_n} :  \big( \fbs(M), \fbs(K) \big) \to \big( \fbs(M_n), \fbs(K_n) \big)$ with $\bign{ \overline{S_n} } \le \lambda$.
Moreover, since for each $x \in M_n$ one has $P_n F_n(x) = \delta_{M_n}(x)$ we have $\n{ S_n(x) - \delta_{M_n}(x) }_{\fbs(M_n)} \le \beta_n$.
If we now consider the operators $J_n \circ \overline{S_n} : \big( \fbs(M), \fbs(K) \big) \to \big( \fbs(M), \fbs(K) \big)$, where $J_n : \big( \fbs(M_n), \fbs(K_n) \big) \to \big( \fbs(M), \fbs(K) \big)$ is the canonical injection, it follows from the above estimates that
the sequence $(J_n \circ \overline{S_n})_n$ converges to the identity of $\fbs(M)$ in the strong operator topology; this proves \eqref{Thm-ext:BAP}.
 \end{proof}
 
 \begin{remark}
In the version for a single space of Theorem \ref{Thm-ext} \cite[Thm. 2.1]{Godefroy-Extensions-and-BAP}, the condition corresponding to \eqref{Thm-ext:alm-ext-R} follows formally from the condition corresponding to \eqref{Thm-ext:alm-ext-X} by specializing to $X=\R$. That is not the case in our version for pairs, since such specialization only gives the BAP for $\fbs(M)$.
 \end{remark}

Our next closely related result
characterizes the BAP for a pair of Lipschitz-free spaces over separable metric spaces, in terms of a Lipschitz version of the principle of local reflexivity respecting subspaces: it is an adaptation to pairs of \cite[Thm. 2]{Godefroy-Ozawa}.
Before the theorem, we will need some terminology including a refinement of the concept of a simultaneous $M$-ideal. The Ando-Choi-Effros Theorem respecting subspaces will not play a direct role here, but some similar ideas will indeed be involved.

Let $M_2 \subset M_1$ be separable metric spaces and $X_2 \subset X_1$ complete metric spaces.
We denote by $\Lip^\lambda(M_1, M_2 ; X_1, X_2)$ the set of all $\lambda$-Lipschitz maps $f : (M_1,M_2) \to \X$.
Let us fix a dense sequence $(x_n)_n$ in $M$ and define a metric $d$ on  $\Lip^\lambda(M_1, M_2 ; X_1, X_2)$ by
$$
d(f,g) = \sum_{n=1}^\infty \min \big\{  d(f(x_n),g(x_n)), 2^{-n} \big\}. 
$$
Observe that $d$ is a complete metric on $\Lip^\lambda(M_1, M_2 ; X_1, X_2)$, whose induced topology is the topology of pointwise convergence.
 
Let $\J \subset \X$ be \pbss{}, and let $Q : X_1 \to X_1/J_1$ be the canonical projection. 
We say that $\J$ is a \emph{simultaneous $M$-ideal with approximate unit} (simultaneous $M$-iwau for short) if there are nets of bounded linear operators $\phi_\alpha : \X \to \J$ and $\psi_\alpha : \X \to \X$ such that $\phi_\alpha(x) \to x$ for every $x \in J_1$, $Q \circ \psi_\alpha = Q$ for all $\alpha$, $\phi_\alpha + \psi_\alpha \to Id_{X_1}$ pointwise, and $\n{\phi_\alpha(x) + \psi_\alpha(y)} \le \max\{ \n{x}, \n{y}\}$ for any $x,y \in X_1$ and all $\alpha$.
 
 \begin{lemma}\label{M-iwau-is-M-ideal}
 If $\J \subset \X$ is a simultaneous $M$-iwau, then it is a simultaneous $M$-ideal.
 \end{lemma}

\begin{proof} 
Let $y_1, y_2, y_3 \in B_{J_1}$, $x \in B_{X_1}$ and $\varepsilon>0$.
Choose $\alpha_0$ large enough so that
$$
\n{x - \phi_{\alpha_0} x - \psi_{\alpha_0}x } \le \varepsilon/2, \qquad \n{ \phi_{\alpha_0}y_i - y_i } \le \varepsilon/2, \quad i=1,2,3
$$
Set $y = \phi_{\alpha_0} x \in J_1$. Then
\begin{align*}
\n{x+y_i-y} = \n{x - \phi_{\alpha_0}x + y_i} &\le \n{x - \phi_{\alpha_0}x -\psi_{\alpha_0}x} + \n{\psi_{\alpha_0}x + \phi_{\alpha_0}y_i} + \n{-\phi_{\alpha_0}y_i + y_i} \\
&\le \varepsilon/2 + \max\{ \n{x}, \n{y_i} \} + \varepsilon/2 \le 1+\varepsilon.
\end{align*}
This proves that $J_1$ is an $M$-ideal in $X_1$ \cite[Thm. I.2.2]{Harmand-Werner-Werner}.
If we assume that $x \in B_{X_2}$ then $y = \phi_{\alpha_0} x \in J_2$, so the same argument gives the desired conclusion by Theorem \ref{thm-geo-charac}.
\end{proof}
 
Notice that our habitual example is in fact an $M$-iwau: if $(E_n,F_n)$ is a paving of a \pbs{} $\X$, then $\big(c_0(E_n),c_0(F_n)\big)$ is a simultaneous $M$-iwau in $\big(c(E_n),c(F_n)\big)$, with
 \begin{align*}
 \phi_k( (x_n)_n ) &= (x_1, \dotsc, x_k, 0,0, \dotsc),\\
 \psi_k( (x_n)_n ) &= (0, \dotsc, 0, x_{k+1}, x_{k+2}, \dotsc).
 \end{align*}
 Notice also that the proof of Lemma \ref{M-iwau-is-M-ideal} did not make use of the condition $Q \circ \psi_\alpha = Q$; the importance of that condition will be apparent in the next Lemma. One important consequence of Lemma \ref{M-iwau-is-M-ideal} is that the equality $Q \circ \psi_\alpha = Q$ not only holds as a mapping $X_1 \to X_1/J_1$, but in fact as a mapping $\X \to (X_1/J_1,X_2/J_2)$.
The proof of \cite[Lemma 1]{Godefroy-Ozawa} now gives, verbatim, the following.
 
 \begin{lemma}\label{Lemma-Arveson-kind-of-ACE}
 Let $\J$ be a simultaneous $M$-iwau in $\X$, and $M_2 \subset M_1$ separable metric spaces. Then for every $\lambda \ge 1$ the set
 $$
 \{ Q \circ f \st f \in \Lip^\lambda(M_1, M_2 ; X_1, X_2) \} \subset \Lip^\lambda(M_1, M_2 ; X_1/J_1, X_2/J_2)
 $$
is closed under the topology of pointwise convergence.
 \end{lemma}
 
 The proof of \cite[Thm. 2]{Godefroy-Ozawa} can now be adapted in a straightforward manner to prove the Theorem below. As would be expected, the part of the proof that uses the principle of local reflexivity now requires the version respecting subspaces \cite{Oja-PLR-respecting-subspaces}.
 
 \begin{theorem}
 Let $M_2 \subset M_1$ be separable metric spaces, and $\lambda \ge 1$.
 Then the pair $\big( \fbs(M_1), \fbs(M_2) \big)$ has the $\lambda$-BAP if and only if the following holds:
 for every \pbs{} $\X$ and any $f \in \Lip^1(M_1,M_2; X_1^{**},X_2^{**})$, there is a net in $\Lip^\lambda(M_1,M_2; X_1,X_2)$ which converges to $f$ in the pointwise-weak$^*$ topology.
 \end{theorem}

 Recently, a characterization of the BAP for the free space over a compact metric space has appeared in \cite[Thm. 2.19]{Ambrosio-Puglisi}.
Their proof actually follows from general principles, which we illustrate by proving a version for pairs. 
In the language of  \cite{Ambrosio-Puglisi}, condition \eqref{AmPu-2} below could be paraphrased as the existence of an ``asymptotic simultaneous $\lambda$-random projection''.
 
 \begin{theorem}
  Let $K \subset M$ be compact metric spaces.
 Let $(K_n)_n$ and $(M_n)_n$ be sequences of finite subsets of $K$, respectively $M$, with $K_n \subseteq M_n$, $\bigcup_n K_n$ dense in $K$ and $\bigcup_n M_n$ dense in $M$.
The following are equivalent:
 \begin{enumerate}[(i)]
 \item\label{AmPu-1} The pair $\big( \fbs(M), \fbs(K) \big)$ has the $\lambda$-BAP.
 \item\label{AmPu-2} For every $n \in \N$, there exists $\nu_n : X \to \fbs(M_n)$ such that $\nu_n(K) \subset \fbs(K_n)$ and
 	\begin{enumerate}[(a)]
	\item $\lim_n \n{ \nu_n(x) - \delta_X(x) }_{\fbs(M)} = 0$ for every $x \in \bigcup_n M_n$.
	\item $\nu_n$ is $\lambda$-Lipschitz.
	\end{enumerate}
 \end{enumerate}
 \end{theorem}
 
\begin{proof}
\eqref{AmPu-1} $\Rightarrow$ \eqref{AmPu-2}:
Consider $E_n = \fbs(M_n)$ and $F_n = \fbs(K_n)$.
Theorem \ref{thm-BAP-equivalent-to-simultaneous-splitting} implies the existence of a sequence of linear maps $T_n : \fbs(M) \to \fbs(M_n)$ such that $\lim_n \n{ T_n \mu - \mu  } = 0$ for all $\mu \in \fbs(M)$, $\n{T_n} \le \lambda$ and $T_n(\fbs(K)) \subseteq \fbs(K_n)$. Taking $\nu_n = \restr{T_n}{M}$ clearly gives \eqref{AmPu-2}.

\eqref{AmPu-2} $\Rightarrow$ \eqref{AmPu-1}
Let $S_n : \fbs(M) \to \fbs(M)$ be the linearization of $\nu_n : M \to \fbs(M)$.
Then $(S_n)$ is a sequence of finite-rank bounded linear maps $\fbs(M) \to \fbs(M)$ of norm at most $\lambda$, that converges to the identity in the weak operator topology and leaves $\fbs(K)$ invariant.
By standard arguments, the pair $\big( \fbs(M), \fbs(K) \big)$ has the $\lambda$-BAP.
\end{proof}

\bibliographystyle{amsalpha}

\def\cprime{$'$}
\providecommand{\bysame}{\leavevmode\hbox to3em{\hrulefill}\thinspace}
\providecommand{\MR}{\relax\ifhmode\unskip\space\fi MR }
\providecommand{\MRhref}[2]{%
  \href{http://www.ams.org/mathscinet-getitem?mr=#1}{#2}
}
\providecommand{\href}[2]{#2}

\end{document}